\documentclass[11pt]{article}
\usepackage{amsmath,amsfonts,amssymb,amsthm}
\usepackage{mathrsfs}
\usepackage{latexsym}
\usepackage[english]{babel}

\usepackage{graphicx}
\usepackage[usenames,dvipsnames]{xcolor}

\renewenvironment{proof}[1][\proofname]{{\bfseries #1.} }{\qed}

\def\Cov{{\rm Cov\,}}

\newcommand{\Var}{{\rm Var}}

\newcommand{\eps}{\varepsilon}

\def\authors#1{{ \begin{center} #1 \vspace{0pt} \end{center} } \smallskip}
\def\institution#1{{\sl \begin{center} #1 \vspace{0pt} \end{center} } }
\def\inst#1{\unskip $^{#1}$}
\def\title#1{{\huge\bf  \begin{center} #1 \vspace{0pt} \end{center}  } \smallskip}

\def\paref#1{(\ref{#1})}

\newtheorem{theorem}{Theorem}[section]

\newtheorem{proposition}[theorem]{Proposition}

\newtheorem{lemma}[theorem]{Lemma}
\newtheorem{corollary}[theorem]{Corollary}
\newtheorem{definition}[theorem]{Definition}

\newtheorem{remark}[theorem]{Remark}
\begin{document}
	
	\date{}
	
	\title{\sc Laguerre Expansion for Nodal Volumes\\ and Applications}
	\authors{\large Domenico Marinucci\inst{1}, Maurizia Rossi\inst{2}, Anna Paola Todino\inst{3}}
	\institution{\inst{1}Dipartimento di Matematica, Universit\`a di Roma ``Tor Vergata"\\
		\inst{2}Dipartimento di Matematica e Applicazioni, Universit\`a di Milano-Bicocca\\
		\inst{3}Dipartimento di Scienze e Innovazione Tecnologica, Universit\`a del Piemonte Orientale}
	
	\begin{abstract}
		
	We investigate the nodal volume of random hyperspherical harmonics $\lbrace T_{\ell;d}\rbrace_{\ell\in \mathbb N}$ on the $d$-dimensional unit sphere ($d\ge 2$).  We exploit  an orthogonal expansion in terms of Laguerre polynomials;  this representation entails a drastic reduction in the computational complexity and allows to prove \emph{isotropy} for chaotic components, an issue which was left open in the previous literature. As a further application, we establish our main result, i.e., variance bounds for the nodal volume in any dimension;  for $d\ge 3$ and as the eigenvalues diverge (i.e., as $\ell\to +\infty$), we obtain the upper bound $O(\ell^{-(d-2)})$ (that we conjecture to be \emph{sharp}).  As a consequence, we show that the so-called Berry's cancellation phenomenon holds in any dimension: namely, the nodal variance is one order of magnitude smaller than the variance of the volume of level sets at any non-zero threshold, in the high-energy limit.

%for the nodal variance, thus confirming Berry's cancellation phenomenon. For $d=2$, the exact asymptotic law is known to be  $\frac{1}{32}\log \ell$ (I. Wigman, \emph{Comm. Math. Phys.}, 2010).

		\smallskip
		
		
		\smallskip
		
		
	\end{abstract}

\today

\

\section{Introduction}

%In the last decades a growing interest has been devoted to geometric and topological properties of random fields on manifolds, also in view of applications in several areas, such as Cosmology, Brain Imaging, Chemestry and %Climate Science.

In the last decade a lot of efforts have been devoted to the investigation of the geometric and topological properties of random eigenfunctions on the unit $d$-dimensional sphere $\mathbb{S}^d$, for $d\ge 2$, in particular their nodal length (in dimension 2), or more generally their nodal volumes (in higher dimensions), see e.g. \cite{M23,W23} for some recent reviews of this literature. While the computation of expected values is straightforward, asymptotic variances (in the high-energy regimes) are more challenging. For $d=2$, the exact asymptotic order for the variance is known to be $\frac{1}{32}\log \ell$; this result was first proved by \cite{W10} with a direct computation based on Kac-Rice Formula. Later, the result was extended to a quantitative central limit theorem by \cite{MRW20}; the main tool in their approach is the Wiener-It\^o chaotic decomposition of the nodal length, which allows this quantity to be expressed in terms of linear combinations of Hermite polynomials evaluated in the eigenfunctions themselves and in their partial derivatives. This decomposition turns out to be especially convenient both for the computation of the asymptotic variance and to establish (quantitative) limit theorems, exploiting the well-known Stein-Malliavin approach \cite{NP12}; for random eigenfunctions, this idea had been earlier exploited on the 2-dimensional standard flat torus (i.e., for arithmetic random waves) in \cite{MPRW16} and at about the same time for planar eigenfunctions (i.e., Berry's random waves) in \cite{NPR19}.

In this paper we investigate the nodal volume $\lbrace \mathcal L_{\ell;d-1}\rbrace_{\ell\in \mathbb N}$ of Gaussian Laplace eigenfunctions (random hyperspherical harmonics) $\lbrace T_{\ell;d}\rbrace_{\ell\in \mathbb N}$ on $\mathbb S^d$ for $d\ge 2$. By nodal volume we mean the $d-1$-Hausdorff measure of the zero locus of the field, which is, almost surely, a smooth submanifold of codimension $1$. In principle, to address this issue in any dimension one could follow the same approach as done in dimension 2 by \cite{MPRW16,NPR19,MRW20}; that is, it is possible to compute the projection of the nodal volumes into Wiener chaoses of order $q$, and then study their asymptotic behaviour as the corresponding eigenvalues diverge. A more careful analysis of this approach shows, however, how it becomes practically unfeasible for larger dimensions.

 Indeed, the computation of the chaotic component of order $q$ requires to take care of the projection of the nodal volume on Hermite polynomials evaluated in the eigenfunctions themselves and their $d$ partial derivatives, together with all cross products such that the corresponding degrees of polynomials sum to $q$. As we shall discuss below, it is easy to check that this gives $(d+1)^q$ different terms; although many of these summands can be shown to be null by symmetry arguments, it is obvious that the corresponding expressions become completely unmanageable even for small values of $d$ and $q$ (for instance, for $d=q=4$ we should take care of 625 terms). It is then natural to wonder whether a drastic simplification of these expressions could be feasible, in order to make their asymptotic analysis possible.

 A different route is suggested by some geometric considerations. In fact, as we shall discuss below the Kac-Rice formula allows to express the nodal volumes in terms of the eigenfunctions and \emph{the norm of their gradient}, rather than the single values of their partial derivatives. It is then natural to look for more compact expressions in terms of just these two scalar quantities: the eigenfunctions $T_{\ell;d}$ and the norm of their gradient $\|\nabla T_{\ell;d}\|$. Of course, the standard Hermite expansion must be implemented in terms of Gaussian variables and processes, and  $\|\nabla T_{\ell;d}\|$ does not fit in this framework, even if the eigenfunctions are Gaussian; this is why we need to resort to Laguerre polynomials. 
 
  More precisely, our first main result (Theorem \ref{thm:laguerre}) concerns a formula for the chaotic decomposition %found in \cite{MPRW16, C19, MRW20}
of the nodal volume:
\begin{equation}\label{chaosexp}
\mathcal L_{\ell;d-1}=\mathbb E[\mathcal L_{\ell;d-1}]+\sum_{q=2}^{+\infty} \mathcal L_{\ell;d-1}[2q].
\end{equation}
 Our formula involves \emph{Laguerre polynomials} only, in particular we can write $\mathcal L_{\ell;d-1}[2q]$ as an integral over the hypersphere of a bivariate polynomial, say $p_{2q}$, evaluated at the field $T_{\ell;d}$ and the norm of its gradient $\|\nabla T_{\ell;d}\|$; $p_{2q}$ admits a simple expansion as the sum of $2q$ terms, each of them the product of two Laguerre polynomials, the first computed on the squared eigenfunction, the second on the squared norm of the gradient. As a further characterization, we show that $p_{2q}$ can also be expressed as the sum of $2q$ terms, now taken to be products of Hermite polynomials (evaluated on the eigenfunctions) and Laguerre (evaluated on the squared norm of the gradient). As anticipated, the cardinality of these terms does \emph{not} depend on the dimension of the hypersphere, although the order of the Laguerre polynomials does.
 
It is worth noticing that in \cite[Remark 2.5]{MR21} it has been proved that the \emph{fourth} chaotic component of nodal intersections of Gaussian Laplacian eigenfunctions (arithmetic random waves) on the $3$-dimensional standard flat torus against a fixed reference surface (with nowhere zero curvature) can be written as the integral of a bivariate polynomial of degree $4$ evaluated at the field restricted to the surface and the norm of its gradient. Moreover, in \cite{Not23} the author extended results of \cite{Tha93}  to obtain a compact formula, analogous to ours, for the Wiener-It\^o chaos expansion of nodal intersections for (at most $3$) i.i.d. copies of arithmetic random waves, still in dimension $3$. We stress that our proof for Theorem \ref{thm:laguerre} is of different nature than those for \cite[(3.54)]{Not23}.
 
We consider a pair of applications of independent interest  
of the neat analytic expressions achieved by the Laguerre expansion in Theorem \ref{thm:laguerre}. Firstly, we note that our formula establishes the isotropy of the integrand functions in the  chaotic components; indeed the random field
 $$
 \mathbb S^d\ni x\mapsto p_{2q}(T_{\ell;d}(x), \| \nabla T_{\ell;d}(x)\|)
 $$ is obviously \emph{isotropic} for every $q$. As a consequence, the variance of \emph{each} chaotic component  can be expressed, up to a constant factor, as a one-dimensional integral (over the ``meridian'' of the hypersphere) of cross moments of Gegenbauer polynomials (the covariance kernel of random hyperspherical harmonics) and their derivatives up to order two.

 As a second application, we study the fluctuations of the nodal volume for $d\ge 3$ around its mean, in the high-energy limit, i.e., as $\ell\to +\infty$. (For $d=2$, as mentioned earlier the exact asymptotic law is known to be  $\frac{1}{32}\log \ell$, see \cite{W10}.)  It is known that $\mathbb E[\mathcal L_{\ell;d-1}]= c_d \sqrt{\ell(\ell+d-1)}$, where $c_d$ is an explicit positive constant only depending on $d$, and $\ell(\ell+d-1) =:E_{\ell;d}$ is the $\ell$-th eigenvalue of the Laplace-Beltrami operator on the $d$-dimensional hypersphere.
We are able to establish the upper bound
	\begin{equation}\label{varbound}
	\Var(\mathcal L_{\ell;d-1}) = O\left ( \ell^{-(d-2)}\right ), \qquad \ell \to +\infty,
	\end{equation}
	that we conjecture to be \emph{sharp}.  We also show that for any non-zero threshold $u \neq 0$, the variance of the volume of the level set $\lbrace x\in \mathbb S^d:T_{\ell;d}(x) = u\rbrace$ is of \emph{exact} order $\ell^{-(d-3)}$, as the eigenvalues diverge; the two results taken together confirm that the Berry's cancellation phenomenon \cite{Ber02} holds for $\mathbb S^d$, $d\ge 3$, namely the ratio between the nodal variance and the variance of the boundary volume at any other threshold converges to zero, as the eigenvalues diverge. 
	
	This asymptotic behaviour can be interpreted in terms of the disappearance of the second chaotic component in the chaotic expansion of the $u$-volume for $u=0$, a phenomenon that had been first noticed in \cite{MPRW16}  for the nodal length of arithmetic random waves. Since then, cancellation phenomena have been proved for a wide class of geometric functionals in \cite{C19, NPR19, MRW20, Tod20, Vid21, Not21, Tod23}, among others.
	
	Our upper bound in \eqref{varbound} improves the result $\Var(\mathcal L_{\ell;d-1}) = O \left ( \ell^{-(d-5)/2} \right )$ established in \cite{Wig09} -  the latter is nevertheless enough to guarantee a law of large numbers for $\mathcal L_{\ell;d-1}/\sqrt{E_{\ell;d}}$, consistently with Yau's conjecture \cite{Yau}.
	
	As far as the second-order fluctuations are concerned, in dimension two they have been shown to be asymptotically Gaussian as previously recalled; the main tool to establish this central limit theorem is the asymptotic equivalence which has been established as $\ell\to +\infty$, in $L^2(\mathbb P)$,  between the nodal length and the so-called sample trispectrum, that is, $\sqrt{E_{\ell;2}} \int_{\mathbb S^2} H_4(T_{\ell;2}(x))\,dx$ where $H_4$ is the fourth Hermite polynomial. In other words, in the chaotic decomposition \eqref{chaosexp}  in dimension $d=2$ the fourth component (corresponding to $q=2$ in our notation) gives the leading term, indeed it has the same asymptotic variance as the total nodal length whereas all the other components are  asymptotically negligible. In dimension $d\ge 3$ instead, we conjecture that \emph{all} chaotic components will share the same order of magnitude and no term will be negligible. Furthermore, we conjecture that the fluctuations are Gaussian also in this case; we leave both these topics for future research.
	
	  Finally, we stress that the Laguerre expansion introduced in this paper for nodal volumes is valid in much greater generality than for random spherical eigenfunctions; indeed, a careful examination of the proof reveals that only pointwise indipendence of the random field with its gradient is required, a property that holds for every Gaussian \emph{isotropic random field}. In this sense, the results in this paper may open further avenues for research on excursion volumes in much broader frameworks than random (hyper)spherical eigenfunctions.

 \subsection{Acknowledgments}

We are grateful to GNAMPA-INdAM Project 2022 \emph{Propriet\'{a} e teoremi limite per funzionali di campi Gaussiani}, to MUR Department of Excellence Project 2023-2027 \emph{MatModTov} and to MUR Prin 2023-2025 \emph{Grafia} for financial support.

\section{Background and notation}

\subsection{Random hyperspherical harmonics}\label{subsec:sphericalharmonics}

For $d\ge 2$, let $\mathbb S^d\subset \mathbb R^{d+1}$ denote the $d$-dimensional unit sphere with the round metric. In standard spherical coordinates $ \theta_i\in [0,\pi]$, $i=1,\dots, d-1$, $\varphi\in [0,2\pi]$, the metric is given by
\begin{equation}\label{metricsphere}
 d\theta_1^2 + \sum_{i=2}^{d-1} \left (\sum_{j=1}^{i-1}\sin^2 \theta_j \right ) d\theta_i^2 + \left (\sum_{j=1}^{d-1}\sin^2 \theta_j \right ) d\varphi^2.
\end{equation}
Let $\Delta_d$ be the corresponding Laplace-Beltrami operator; we consider the Helmholtz equation
$
\Delta_d f + E f =0
$
where $E\ge 0$ and $f:\mathbb S^d \to \mathbb R$. The eigenvalues of $-\Delta_d$ are of the form $E=E_{\ell;d}:= \ell(\ell+d-1)$ for $\ell\in \mathbb N$,
and the $\ell$-th eigenspace has dimension
\begin{equation}\label{dimeig}
\eta_{\ell;d} := \frac{2\ell+d-1}{\ell} \binom{\ell+d-2}{\ell-1} \sim \frac{2}{(d-1)!}\ell^{d-1} \quad \mbox{ as } \ell \to \infty.
\end{equation}
We choose the hyperspherical harmonics $\lbrace Y_{\ell,m;d}\rbrace_{m=1,\dots, \eta_{\ell;d}}$, see \cite[Section 9.3]{VK93}, as an orthonormal basis of the $\ell$-th eigenspace
$$
\Delta_{d}Y_{\ell,m;d}+E_{\ell;d} Y_{\ell,m;d}=0;
$$
they are homogeneous polynomials in $d+1$ variables restricted to $\mathbb S^d$.

The $\ell$-th random hyperspherical harmonics $T_{\ell;d}$ are defined as follows:
$$
T_{\ell;d}(x) := %\sqrt{\frac{|\mathbb{S}^d|}{\eta_{\ell,d}}}
\sum_{m=1}^{\eta_{\ell;d}} a_{\ell,m;d}Y_{\ell,m;d}(x),\qquad x\in \mathbb S^d,
$$
where $\lbrace a_{\ell,m;d}\rbrace_{m=1,\dots, \eta_{\ell;d}}$ are i.i.d. centered Gaussian random variables (on some suitable probability space $(\Omega, \mathcal F, \mathbb P)$) with variance
$
\Var(a_{\ell,m;d})=\frac{\mathcal H^d(\mathbb{S}^d)}{\eta_{\ell;d}},
$
and $\mathcal H^d(\mathbb{S}^d)$
is the measure of the $d$-sphere. The random field $T_{\ell;d}$ is centered Gaussian with covariance kernel
$$
\mathbb E[T_{\ell;d}(x)T_{\ell;d}(y)]= \frac{\mathcal H^d(\mathbb{S}^d)}{\eta_{\ell;d}} \sum_{m=1}^{\eta_{\ell;d}} Y_{\ell,m;d}(x) Y_{\ell,m;d}(y)= {G}_{\ell;d}(\cos d(x,y)) \qquad x,y \in \mathbb{S}^d,
$$
where ${G}_{\ell;d}$ is the $\ell$-th Gegenbauer polynomial normalized such that ${G}_{\ell;d}(1)=1$, and $d(x,y)=\arccos\langle x, y\rangle$ denotes the geodesic distance between $x$ and $y$ in $\mathbb S^d\subset \mathbb R^d$. In particular, $T_{\ell;d}$ is isotropic, i.e., invariant in law with respect to the action of the special orthogonal group $SO(d+1)$, and has unit variance.

Recall that
$
{G}_{\ell;d}  =\alpha_\ell^{-1} P_\ell^{(d/2-1,d/2-1)},
$
where $\alpha_\ell:= \binom{\ell+d/2-1}{\ell}$ and
$\lbrace P_\ell^{(d/2-1,d/2-1)}\rbrace_{\ell \in \mathbb N}$ is the family of Jacobi polynomials
$$
P_\ell^{(d/2-1,d/2-1)}(t)= \frac{1}{2^\ell} \sum_{k=0}^{\ell} \binom{\ell+d/2-1}{k} \binom{\ell+d/2-1}{\ell-k} (t-1)^{\ell-k} (t+1)^{k},\qquad t\in [-1,1],
$$
which is orthogonal on $[-1,1]$ with respect to the weight $(1-t^2)^{d/2 -1}$.
Hilb's asymptotic formula \cite[Theorem 8.21.12]{szego} entails that, for any $C>0$, $\varepsilon >0$, uniformly for $\theta\in [0,\pi-\varepsilon]$ and $\ell\ge 1$,
\begin{align}\label{hilbs}
(\sin \theta)^{d/2-1} {G}_{\ell;d}(\cos \theta) = \frac{2^{d/2-1}}{{\ell + d/2-1 \choose \ell}} \frac{\Gamma(\ell+d/2)}{L^{d/2-1}\ell!} \sqrt{\frac{\theta}{\sin \theta}} J_{d/2-1}( L\theta) + \delta(\theta),
\end{align}
where $L:= \ell + (d-1)/2$, $J_{d/2-1}$ is the Bessel function of the first kind of order $d/2-1$, and
\begin{equation*}
\delta(\theta) \ll \begin{cases}
\theta^{1/2} \ell^{-3/2},\qquad &\theta > C/\ell,\\
\theta^2, \qquad &0< \theta < C/\ell.
\end{cases}
\end{equation*}
 (The kernel $\mathbb R^d \times \mathbb R^d \ni (x,y) \mapsto J_{d/2-1}(\|x-y\|)$  is positive definite \cite{MN23} for every $d\ge 2$.)

In particular, for $d=2$, the scaling limit of random spherical harmonics is the so-called Berry's random wave model \cite{Ber77}, which is conjectured to be a universal
model for high-energy two-dimensional eigenfunctions of generic classically chaotic billiards.

\subsection{Nodal volumes: prior work}

The nodal set
$
T_{\ell;d}^{-1}(0) := \lbrace x\in \mathbb S^d : T_{\ell;d}(x)=0\rbrace
$
is a.s. a smooth hypersurface; we are interested in the asymptotic distribution, as $\ell\to +\infty$, of its $d-1$--Hausdorff measure
$$
\mathcal L_{\ell;d-1} := \mathcal H^{d-1}(T_{\ell;d}^{-1}(0)).
$$
 In 1985 B\'erard \cite{Be85} computed
the expectation to be proportional (by a constant factor only depending on $d$) to $\sqrt{E_{\ell;d}}$, more precisely
\begin{equation}\label{meanvol}
\frac{\mathbb E[\mathcal L_{\ell;d-1} ]}{\sqrt{E_{\ell;d}}} %= \frac{1}{\sqrt \pi} \sqrt{\frac{1}{d}} \mathcal H^d(\mathbb S^d) \frac{\Gamma \left ( \frac{d+1}{2} \right )}{\Gamma \left ( \frac{d}{2}\right ) }
=\sqrt{\frac{1}{d}} \mathcal H^{d-1}(\mathbb S^{d-1}),
\end{equation}
in accordance with Yau's conjecture,
while in \cite{Wig09} Wigman established an upper bound for the variance of the form
\begin{equation}\label{vard}
\Var(\mathcal L_{\ell;d-1}) = O\left ( \frac{E_{\ell;d}}{\sqrt{\eta_{\ell;d}}}\right ),
\end{equation}
or equivalently $\Var(\mathcal L_{\ell;d-1})= O \left ( \ell^{-(d-5)/2} \right )$. 
In particular, \eqref{meanvol} and \eqref{vard} yield a concentration result for the nodal volume in the high-energy limit: indeed, Markov inequality and \eqref{dimeig} ensure that  $\mathcal L_{\ell;d-1} / \sqrt{E_{\ell;d}}$ converges in probability to a constant only depending on $d$ (that is, the right-hand side of \eqref{meanvol}), as $\ell\to +\infty$.

For $d=2$, in \cite{W10} a finer analysis has been carried out by means of the Kac-Rice formula, leading to the following asymptotic law for the variance of the nodal length
\begin{equation}\label{varIgor}
\Var(\mathcal L_{\ell;1}) \sim \frac{1}{32}\log\sqrt{E_{\ell;2}},\qquad \ell\to +\infty,
\end{equation}
which in particular proves Berry's cancellation phenomenon on the sphere \cite{Ber02}.  Here and in the sequel, by Berry's cancellation phenomenon we mean that the variance for the nodal volume is one order of magnitude smaller, as the eigenvalues diverge, than the ``natural" rate, given by the variance of the volume of the level set $\lbrace x\in \mathbb S^d:T_{\ell;d}(x)=u\rbrace$ at any non-zero threshold $u \neq 0$.

Second order fluctuations for $\mathcal L_{\ell;1}$ were investigated in \cite{MRW20}, where a Central Limit Theorem was established: as $\ell\to +\infty$,
\begin{equation*}
\tilde{\mathcal L}_{\ell;1} := \frac{\mathcal L_{\ell;1}-\mathbb E[\mathcal L_{\ell;1}]}{\sqrt{\Var(\mathcal L_{\ell;1})}} \mathop{\longrightarrow}^d Z\sim \mathcal N(0,1).
\end{equation*}
Actually a stronger result was proved, namely the asymptotic equivalence of the nodal length with the so-called sample trispectrum: as $\ell\to +\infty$,
\begin{equation}\label{asympequiv}
\mathbb E\left [\left  | \tilde{\mathcal L}_{\ell;1} - \tilde{\mathcal M}_{\ell} \right |^2\right ] = O\left (\frac{1}{\log \ell} \right ),
\end{equation}
where, denoting by $H_4(t)=t^4 - 6t^2 + 3$, $t\in \mathbb R$ the $4$-th Hermite polynomial,
\begin{equation*}
\mathcal M_\ell := -\frac{1}{4} \sqrt{\frac{E_{\ell;2}}{2}} \frac{1}{4!}\int_{\mathbb S^2} H_4(T_{\ell;2}(x))\,dx,\qquad \tilde{\mathcal M}_\ell := \frac{\mathcal M_\ell }{\sqrt{\Var(\mathcal M_\ell)}}.
\end{equation*}

\section{Motivations and main results}

As mentioned earlier, the first goal in this paper is to achieve a drastic simplification in the so-called Hermite chaos expansion of the nodal volume in any dimension. This result has several important applications; to mention one, for $d\ge 3$, no sharp asymptotics (as $\ell\to +\infty$) for the variance of the nodal volume $\mathcal L_{\ell;d-1}$ is available (and there is no information on its limiting distribution); we will show later how to exploit our result to obtain upper bounds, that we strongly conjecture to be sharp. We shall also show how our decomposition allows to establish Berry's cancellation result in any dimension;
there are, however, some marked differences with the two-dimensional case, which we shall discuss below.

The nodal volume can be formally written as
\begin{equation*}
\mathcal L_{\ell;d-1} = \int_{\mathbb S^d} \delta_0(T_{\ell;d}(x))\| \nabla T_{\ell;d}(x)\|\,dx,
\end{equation*}
where $\delta_0$ denotes the Dirac mass at zero, and $\nabla$ the gradient operator.  More precisely, let us define for $\varepsilon >0$
\begin{eqnarray*}
	\mathcal{L}^\varepsilon_{\ell;d-1} := \int_{\mathbb{S}^d} \frac{1}{2\varepsilon}1_{[-\varepsilon, \varepsilon]}(T_{\ell;d}(x)) ||\nabla T_{\ell;d}(x)||\, dx, %= \sqrt{\frac{\ell(\ell+d-1)}{d}} \int_{\mathbb{S}^d}\delta_0(T_\ell(x)) ||\tilde{\nabla} T_\ell(x)||\, dx
\end{eqnarray*}
then we have the following result.
\begin{lemma}\label{lemconv}
As $\varepsilon \to 0$,
$$
\mathcal{L}^\varepsilon_{\ell;d-1} \longrightarrow \mathcal{L}_{\ell;d-1}
$$
both a.s. and in $L^2(\mathbb P)$.
\end{lemma}
The proof of the a.s. convergence in Lemma \ref{lemconv} follows from the co-area formula, \cite[Lemma 2.2]{Wig09} and \cite[Lemma 2.9]{Wig09}; moreover
 \cite[Lemma 2.11]{Wig09} allows to prove the convergence in $L^2(\mathbb P)$ as an application of the dominated convergence theorem.

Being a square integrable functional of a Gaussian field, the random variable $\mathcal L_{\ell;d-1}$ can be developed into Wiener chaoses, i.e., it can be written as an orthogonal series, converging in $L^2(\mathbb P)$, of the form
\begin{equation}\label{chaostot}
\mathcal{L}_{\ell;d-1} = \mathbb E[\mathcal{L}_{\ell;d-1}]+\sum_{q=1}^{+\infty} \mathcal{L}_{\ell;d-1}[q].
\end{equation}
Loosely speaking, this expansion is based on the fact that Hermite polynomials $\lbrace H_q\rbrace_{q\in \mathbb N}$ form an orthogonal basis for the space of square integrable functions on the real line w.r.t. the Gaussian density \cite{NP12}. Recall that $H_0\equiv 1$ and, for $q\ge 1$,
$$
H_q(t) = (-1)^q e^{t^2/2} \frac{d^q}{dt^q} e^{-t^2/2},\qquad t\in \mathbb R.
$$
The first few Hermite polynomials are $H_1(t)=t$, $H_2(t)=t^2-1$, $H_3(t)=t^3-3t$, $H_4(t)=t^4 - 6t^2 + 3$. More precisely, for \eqref{chaostot} we have the following result.
\begin{proposition}\label{propChaosHermite}
The Wiener-It\^o chaos decomposition \eqref{chaostot} of $\mathcal L_{\ell;d-1}$ is
\begin{eqnarray*}
\mathcal L_{\ell;d-1}[2q+1] &=& 0, \textnormal{ for } q\ge 0,\\
\mathcal L_{\ell;d-1}[2] &=& 0,\\
\mathcal L_{\ell;d-1}[2q] &= & \sqrt{ \frac{E_{\ell;d}}{d} }\sum_{p=0}^{q} \frac{\beta_{2q-2p}}{(2q-2p)!} \\
&&\times \!\!\!\! \sum_{\substack{s=(s_1,\dots, s_d) \in \mathbb{N}^d \\ s_1+\dots+s_d=p}} \frac{\alpha_{2s_1,\dots,2s_d}}{(2s_1)!\dots(2s_d)!} \int_{\mathbb{S}^d} H_{2q-2p}(T_{\ell;d}(x)) \prod_{j=1}^{d} H_{2s_j}(\tilde{\partial}_j T_{\ell;d}(x))\,dx,
\end{eqnarray*}
for  $q\ge 2$,
where $\tilde{\partial}_j T_{\ell;d}(x)=\frac{\partial_j T_{\ell;d}(x)}{\sqrt{\ell(\ell+d-1)/d}} $,
$\beta_{2q-2p}= \frac{1}{\sqrt{2\pi} } H_{2q-2p}(0)$ and, for $s=(s_1,\dots, s_d)\in \mathbb{N}^d$,
\begin{eqnarray}
\alpha_{2s_1,\dots,2s_d}= &&\sum_{i=0}^{\infty} \frac{1}{i!2^i} \frac{\sqrt{2}\Gamma(\frac{d}{2}+i+\frac{1}{2})}{\Gamma( \frac{d}{2}+i)} \nonumber \\
&\times& \sum_{\substack{j_1+\dots+j_d=i\\ j_1\le s_1, \dots, j_d\le s_d}} \binom{i}{j_1,\dots,j_d} \frac{(-1)^{s_1-j_1+\dots+s_d-j_d}(2s_1)!\dots(2s_d)!}{(s_1-j_1)!\dots(s_d-j_d)! 2^{ s_1-j_1+\dots+s_d-j_d}}.\nonumber \\
\label{adef}
\end{eqnarray}
\end{proposition}
We anticipate that the above mentioned Berry's cancellation phenomenon is equivalent to the disappearence of the second chaotic component, that is, $\mathcal L_{\ell;d-1}[2]=0$.

The proof of Proposition \ref{propChaosHermite} follows from Lemma \ref{lemconv} and an adaptation of the proof of  \cite[Lemma 4.1]{C19} to the spherical case, hence we omit the details for the sake of brevity. A careful inspection of the previous formulae and a standard combinatorial argument reveals that there are $(d+1)^{2q}$ summands in the chaos of order $2q$.

 Our first main result is a reformulation of the Wiener-It\^o development of $\mathcal L_{\ell;d-1}$ in Proposition \ref{propChaosHermite} in terms of Laguerre polynomials.

 \subsection{Laguerre expansion}

The closed form of Laguerre polynomials $\lbrace L_n\rbrace_{n\in \mathbb N}$ is
	$$
	L_n(t)=\sum_{i=0}^{n} \binom{n}{i} (-1)^i  \frac{t^i}{i!},\qquad t\in \mathbb R.
	$$
The first Laguerre polynomials are
	$L_0\equiv 1,
	L_1(t)=-t+1,
	L_2(t)=\frac{1}{2}(t^2-4t+2)$.
\begin{definition}\label{defLaguerre}
	For $\alpha\in \mathbb R$, the generalized Laguerre polynomials $\lbrace L_n^{(\alpha)}\rbrace_{n\in \mathbb N}$ are defined as
	$$L_0^{(\alpha)}\equiv1,\qquad L_1^{(\alpha)}(t):=1+\alpha-t, \quad t\in \mathbb R,$$
	and for any $k\in \mathbb N$, $k \geq 1$,
	$$
	L_{k+1}^{(\alpha)}(t):=\frac{(2k+1+\alpha-t)L_k^{(\alpha)}(t)-(k+\alpha)L_{k-1}^{(\alpha)}(t)}{k+1},\qquad t\in \mathbb R.
	$$
\end{definition}
The closed form for the generalized Laguerre polynomials of degree $n$ is
	$$L_n^{(\alpha)}(t)=\sum_{i=0}^{n} (-1)^i \binom{n+\alpha}{n-i} \frac{t^i}{i!},\qquad t\in \mathbb R.$$
The simple Laguerre polynomials are the special case $\alpha=0$ of the generalized Laguerre polynomials: $L_n^{(0)}=L_n$.
The first few generalized Laguerre polynomials are:
$
L_0^{(\alpha)}\equiv 1,
L_1^{(\alpha)}(t)=-t+(\alpha+1),
L_2^{(\alpha)}(t)=\frac{t^2}{2}-(\alpha+2)t+\frac{(\alpha+1)(\alpha+2)}{2}.
$

We are in a position to state the first main result of this paper.
%The value at 0 is $$L_n^{(\alpha)}(0)=\binom{n+\alpha}{n}=\frac{\Gamma(n+\alpha+1)}{n! \Gamma(\alpha+1)}$$
 \begin{theorem}\label{thm:laguerre}
 For $q\in \mathbb N$, $q\ge 2$,
	 	\begin{eqnarray*}
	\mathcal{L}_{\ell;d-1}[2q]
	 		&=& \sqrt{ \frac{E_{\ell;d}}{d} } \int_{\mathbb{S}^d} p_{2q}(T_{\ell;d}(x), ||\widetilde{\nabla} T_{\ell;d}(x)||)\,dx,
	 	\end{eqnarray*}
	 	where for $r,t\in \mathbb R$
		\begin{equation}\label{pol2q}
		p_{2q}(r,t) = \sum_{p=0}^{q} \frac{\beta_{2q-2p}}{(2q-2p)!} \sqrt{2}C(d,p) 2^{q-p} (-1)^{q-p}(q-p)!
	 		 L_{q-p}^{(-1/2)}\left(\frac{r^2}{2}\right) L_{p}^{(d/2-1)}\left(\frac{t^2}{2}\right),
					\end{equation}
		and
	 	\begin{equation}\label{Cdef}
		C(d,p):=\sum_{i=0}^{\infty} (-1)^i \frac{\Gamma(\frac{d}{2}+i+\frac{1}{2})}{\Gamma( \frac{d}{2}+i)} \binom{p}{i} = - \frac{\Gamma(\frac{d+1}{2})\Gamma(\frac{1}{2}(2p-1))}{2\sqrt{\pi} \Gamma(\frac{1}{2}(d+2p))},
		%\textcolor{blue}{=-\frac{1}{\sqrt{\pi}} \frac{\mathcal{H}^{d+2p-1}(\mathbb{ S}^{d+2p-1})}{\mathcal{H}^{d}(\mathbb{ S}^{d}) \mathcal{H}^{2p-2}(\mathbb{ S}^{2p-2})} },
		\end{equation}
		$\lbrace L_n^{(\alpha)}\rbrace_{n\in \mathbb N}$ being the generalized Laguerre polynomials in Definition \ref{defLaguerre}.
	 	 \end{theorem}
		 The proof of Theorem \ref{thm:laguerre} will be given in Section \ref{sec:proofThmLaguerre}.
		 \begin{remark}\label{remThm1}\rm
Theorem \ref{thm:laguerre} entails that every (non null) chaotic component $\mathcal L_{\ell;d-1}[2q]$ (see Proposition \ref{propChaosHermite}) in \eqref{chaostot} of the nodal volume can be written as
the integral over the (hyper)sphere of an explicit bivariate polynomial ($p_{2q}$ in \eqref{pol2q}), evaluated at the field $T_{\ell;d}$ and the norm of its (normalized) gradient $\|\widetilde{\nabla} T_{\ell;d} \|$.
In particular, the spherical random field
$$
\mathbb S^d\ni x\mapsto p_{2q} (T_{\ell;d}(x), \|  \widetilde{\nabla} T_{\ell;d}(x) \| )$$
is \emph{isotropic}.
\end{remark}

\begin{remark} \rm A careful inspection of the proof reveals that it is also possible to establish a \emph{Hermite-Laguerre expansion} for the nodal volume, which takes the following form:
\begin{eqnarray*}
	\mathcal{L}_{\ell;d-1}[2q]&=& \sqrt{ \frac{\ell(\ell+d-1)}{d} }\sum_{p=0}^{q}\frac{\beta_{2q-2p}}{(2q-2p)!}  \sqrt{2} C(d,p)\\&&
	\times \int_{\mathbb{S}^d} H_{2q-2p}(T_{\ell;d}(x)) L_p^{(d/2-1)}\left(\frac{||\tilde{\nabla} T_{\ell;d}(x)||^2}{2}\right) \,dx.
\end{eqnarray*}
As before, we stress that here we obtain a drastic complexity reduction with respect to the standard Hermite expansion, as the number of terms involved grows linearly with $q$ and is constant in the dimension $d$.
\end{remark}

As we mentioned in the Introduction, the previous results are only based on the analytic expression for the expansion of the nodal volume into Hermite polynomials at any fixed point $x \in \mathbb{S}^d$. The coefficients in this expansion are universal for any isotropic Gaussian field - for instance, for arithmetic random fields on the standard flat $d$-dimensional torus $\mathbb{T}^d$ for $d\ge 2$ (see \cite[(3.54)]{Not23} for the case of nodal intersections on the $3$-dimensional standard flat torus). As a consequence, the Laguerre expansion holds in much greater generality than for random eigenfunctions -  other possible applications are left for future research.

\subsection{Variance of the nodal volume}

From Proposition \ref{propChaosHermite} and Theorem \ref{thm:laguerre},
\begin{eqnarray}
\Var(\mathcal L_{\ell;d-1}) &=& \sum_{q=2}^{+\infty} \Var(\mathcal L_{\ell;d-1}[2q])\nonumber \\
&=& \frac{E_{\ell;d}}{d}\sum_{q=2}^{+\infty}   \int_{(\mathbb{S}^d)^2} \mathbb E[p_{2q}(T_{\ell;d}(x), ||\widetilde{\nabla} T_{\ell;d}(x)||)p_{2q}(T_{\ell;d}(y), ||\widetilde{\nabla} T_{\ell;d}(y)||)]\,dx dy.\nonumber \\ \label{doubleint}
\end{eqnarray}
Thanks to Remark \ref{remThm1}, the quantity
$\mathbb E[p_{2q}(T_{\ell;d}(x), ||\widetilde{\nabla} T_{\ell;d}(x)||)p_{2q}(T_{\ell;d}(y), ||\widetilde{\nabla} T_{\ell;d}(y)||)]$ only depends on $d(x,y)$ (and $\ell$, $q$); this simple observation leads to a key simplification. Indeed, in order to evaluate the double integral \paref{doubleint} on $\mathbb S^d$, we may restrict ourselves to points on the ``meridian''
(in spherical coordinates, the set of points such that $\theta_2=\dots = \theta_{d-1}=\varphi =0$), points for which the covariance matrix of the
Gaussian vector $(T_{\ell;d}(x), \widetilde{\nabla} T_{\ell;d}(x), T_{\ell;d}(y), \widetilde{\nabla} T_{\ell;d}(y))$ becomes \emph{sparse} (see \eqref{covmatrix}).
Thus we are able to provide an explicit formula for the variance of \emph{each} chaotic component of the nodal volume in terms of cross moments of Gegenbauer polynomials and their first two derivatives, with the help of the diagram formula for cross moments of Hermite polynomials (see \cite[Lemma 5.2]{CGR22} and \cite[Section 4.3.1]{MP11}).
\begin{proposition}\label{PropVarFormula}
For $q\in \mathbb N$, $q\ge 2$,
\begin{eqnarray*}
&&\Var(\mathcal L_{\ell;d-1}[2q]) = \frac{E_{\ell;d}}{d} \sum_{p,p'=0}^{q} \frac{\beta_{2q-2p}}{(2q-2p)!}\frac{\beta_{2q-2p'}}{(2q-2p')!}
	\frac{(-1)^p\sqrt{2}}{2^p}C(d,p)\frac{(-1)^{p'}\sqrt{2}}{2^{p'}}C(d,p')\cr
	&&\times \sum_{\substack {s_1,s_2,\dots, s_d \in \mathbb{N}\\  s_1+\dots+s_d=p}} \frac{(2s_1)!(2(p'-s_2-\dots-s_d))!}{s_1!(p'-s_2-\dots-s_d)!}\frac{((2s_2)!)^2\cdots ((2s_d)!)^2}{(s_2!)^2 \cdots (s_d!)^2}\times (2q-2p)! (2q-2p')!  \cr
	&&\times 2 \mathcal H^{d}(\mathbb S^d) \mathcal H^{d-1}(\mathbb S^{d-1})\int_{0}^{\pi/2} d\theta \sin^{d-1} \theta \sum_{i=2}^d\frac{\Big (\big ( \frac{\ell(\ell+d-1)}{d} \big )^{-1} G'_{\ell;d}(\cos \theta)\Big )^{2s_{i}}}{(2s_i)!}\cr
	&&\times \sum_{ k\in \mathbb N^{4} : k\in \mathcal A_{q,p,p'}} \frac{G_{\ell;d}(\cos \theta)^{ k_{12}}}{ k_{12}!} \times \frac{(-\big ( \frac{\ell(\ell+d-1)}{d} \big )^{-\frac{1}{2}} (\sin \theta)G'_{\ell;d}(\cos \theta))^{ k_{1, d+3}}}{ k_{1,d+3}!}\cr
	&& \times \frac{( \big ( \frac{\ell(\ell+d-1)}{d} \big )^{-\frac{1}{2}} \sin \theta G'_{\ell;d}(\cos \theta))^{\ k_{2, 3}}}{  k_{2,3}!} \times  \cr
	&&\times \frac{\Big( \big ( \frac{\ell(\ell+d-1)}{d} \big )^{-1} G_{\ell;d}'(\cos\theta)\cos \theta - G_{\ell;d}''(\cos \theta)\sin^2 \theta\Big )^{ k_{3,d+3}}}{ k_{3,d+3}!},
		\end{eqnarray*}
		where
		\begin{align}\label{conditionsA}
		\mathcal A _{q,p,p'}:=\lbrace &k\in \mathbb N^{4}:  k_{1,2} + k_{1,d+3} = 2q-2p,
		 k_{1,2} + k_{2,3} = 2q-2p',\nonumber \\		 &k_{2,3} + k_{3,d+3} = 2s_1, k_{1,d+3} + k_{3,d+3} = 2s'_1\rbrace.
		\end{align}
			\end{proposition} 	
			Note that the conditions of the set \eqref{conditionsA} imply that
			\begin{equation*}
		k_{1,2} + k_{1,d+3} + k_{2,3} + k_{3,d+3} + \sum_{i=2}^{d} 2s_i  = 2q-2p + 2s_1 + \sum_{i=2}^d 2s_i = 2q-2p +2p = 2q.
		\end{equation*} 		
		The proof of Proposition \ref{PropVarFormula} is collected in Section \ref{SecVarFormula}.

A careful investigation of the asymptotic behaviour of cross moments of Gegenbauer polynomials, together with Proposition \ref{PropVarFormula}, leads to the
second main result of this paper, which in particular improves \eqref{vard}.
\begin{theorem}\label{thmvar}
Let $d\ge 3$, then
\begin{equation}\label{mainres}
 \Var(\mathcal L_{\ell;d-1})  = O \left (\ell^{-(d-2)}\right ),\qquad \ell\to +\infty.
\end{equation}
\end{theorem}
We strongly believe the upper bound in \eqref{mainres}  to be \emph{sharp}, and we leave as a topic for future research the exact asymptotic law for the variance.

  Let use denote by	$\mathcal{L}_{\ell;d-1}(u)=\mathcal{H}^{d-1}(T_{\ell;d}^{-1}(u))$ the Hausdorff measure of the excursion (upper-level) set boundary for a generic threshold level $u \in \mathbb{R}$. By an application of  Theorem \ref{thmvar}, we have that the Berry's cancellation phenomenon (\cite{Ber02}) holds on the $d$-dimensional unit sphere, for $d\ge 3$; for the two-dimensional sphere, analogous results were shown earlier by \cite{W10, MRW20}, for the flat torus in dimension two and three by \cite{MPRW16, C19}, and for the plane by \cite{NPR19}. 
\begin{corollary} \label{BerryCancel}
For any $u \neq 0$, as $\ell \rightarrow \infty$ we have
\[
\frac{\Var(\mathcal{L}_{\ell;d-1}(0))}{\Var(\mathcal{L}_{\ell;d-1}(u))}=O(\ell^{-1}).
\]
\end{corollary}
The proof of Corollary \ref{BerryCancel} is postponed to the Appendix \ref{AppendixA}.

As for the two-dimensional case, the heuristic explanation for Berry's cancellation is the fact that the second-order chaos term, whose variance is of exact order $\ell^{-(d-3)}$, disappears in the nodal case $u=0$.
% \begin{remark}\rm
%The result of Theorem \eqref{thmvar} confirms Berry's cancellation phenomenon \cite{Ber02} on the $d$-dimensional unit sphere, for $d\ge 3$. Indeed, one would expect the variance to be of order $E_{\ell;d}/\eta_{\ell;d}$, i.e., $1/\ell^{d-3}$, but it is indeed of smaller order (see \eqref{mainres}). This is equivalent to the disappearence of the second chaotic component in the expansion of the nodal volume (see Proposition \ref{propChaosHermite});
%it has been observed already on the two-dimensional sphere \cite{W10, MRW20}, and on the standard flat torus in dimension two and three \cite{MPRW16, C19}.
% \end{remark}

 \begin{remark}\rm
For $d=2$, our argument would give an alternative proof of \eqref{varIgor} by means of chaotic decomposition.
There is a \emph{marked difference} between the two dimensional case and the higher dimensional setting. Indeed, in dimension two, for $q\ge 3$
\begin{equation*}
\Var(\mathcal L_{\ell;1}[2q]) = o(\Var(\mathcal L_{\ell;1}[4])),\qquad \ell\to +\infty,
\end{equation*}
thus implying $
 \Var(\mathcal L_{\ell;1}) \sim \Var(\mathcal L_{\ell;1}[4])$ as $\ell\to +\infty$ (cf. \eqref{asympequiv}),
while for $d\ge 3$ all (non null) chaotic components share the same order of magnitude with regard to the asymptotic variance.

This is in contrast with the toral case, indeed the variance of the nodal volume of the standard flat torus  is asymptotically equivalent to the variance of its fourth chaotic component not only in dimension $2$ \cite{KKW13, MPRW16} but also in dimension $3$ \cite{BM, C19}. Moreover, a non-Central Limit Theorem holds \cite{MPRW16, C19}.

We strongly believe that the second order fluctuations of the nodal volume $\mathcal L_{\ell;d-1}$ of random hyperspherical harmonics are asymptotically Gaussian also for $d\ge 3$ (for $d=2$ a CLT has been proved in \cite{MRW20}): we leave this open problem as a topic for future research.
\end{remark}

\section{Proof of Theorem \ref{thm:laguerre}}  \label{sec:proofThmLaguerre}

In this section we prove Theorem \ref{thm:laguerre}. To this aim let us start with some technical results. Recall definitions \eqref{adef} and \eqref{Cdef}.
	\begin{lemma}\label{lem:as}
	For $s=(s_1,\dots, s_d)\in \mathbb N^d$ we have
		\[\frac{\alpha_{2s_1,\dots,2s_d}}{(2s_1)!\dots(2s_d)!} %= \frac{(-1)^p\sqrt{2}}{2^p} \frac{1}{s_1!...s_d!}\left(- \frac{\Gamma(\frac{d+1}{2})\Gamma(\frac{1}{2}(2p-1))}{2\sqrt{\pi} \Gamma(\frac{1}{2}(d+2p)}\right)
		=\frac{(-1)^p\sqrt{2}}{2^p} \frac{1}{s_1! \cdots s_d!} C(d,p), \]	
		where $p=s_1+\dots + s_d$.
	\end{lemma}

	\begin{proof}
		Let us consider formula \eqref{adef} and set $s_1+\dots +s_d=p$. Then multiplying and dividing by $s_1! \cdots s_d!$ we observe that
		\begin{align*}
	\alpha_{2s_1,\dots,2s_d}= \frac{(-1)^p\sqrt{2}}{2^p} &\sum_{i=0}^{\infty} (-1)^i\frac{1}{i!} \frac{\Gamma(\frac{d}{2}+i+\frac{1}{2})}{\Gamma( \frac{d}{2}+i)}\\
		& \times \sum_{\substack{j_1+\dots+j_d=i\\ j_1\le s_1, \dots, j_d\le s_d}}\frac{i!}{j_1!\dots j_d!} \frac{s_1!...s_d!}{(s_1-j_1)!\dots(s_d-j_d)! } \frac{(2s_1)!\dots(2s_d)!}{s_1!...s_d!}.
		\end{align*}
		Finally, Vandermonde's identity
		%that is
	%	\[ \binom{n_1+\dots+n_p}{m}= \sum_{k_1+\dots+k_p=m} \binom{n_1}{k_1} \cdots \binom{n_p}{k_p} , \]
	%	(\cite{Askey} page @@@@),
		 entails that
				\[\frac{\alpha_{2s_1,\dots,2s_d}}{(2s_1)!\dots(2s_d)!}= \frac{(-1)^p\sqrt{2}}{2^p} \frac{1}{s_1!...s_d!} \sum_{i=0}^{\infty} (-1)^i \frac{\Gamma(\frac{d}{2}+i+\frac{1}{2})}{\Gamma( \frac{d}{2}+i)} \binom{p}{i},\]
		thus concluding the proof 	of Lemma \ref{lem:as} recalling \eqref{Cdef}.
		
	\end{proof}

Proposition \ref{propChaosHermite} together with Lemma \ref{lem:as} readily implies the following result, whose proof is omitted for brevity sake.
	
	\begin{lemma}\label{th:hermite} For $q\in \mathbb N$, $q\ge 2$,
		\begin{eqnarray*}
			\mathcal{L}_{\ell;d-1}[2q]	&=& 	 \sqrt{ \frac{\ell(\ell+d-1)}{d} }\sum_{p=0}^{q} \frac{\beta_{2q-2p}}{(2q-2p)!}  \frac{(-1)^p\sqrt{2}}{2^p} C(d,p)\\&&
			\times \int_{\mathbb{S}^d} H_{2q-2p}(T_{\ell;d}(x)) \sum_{s \in \mathbb{N}^d \substack s_1+\dots+s_d=p} \frac{1}{s_1! \cdots s_d!} \prod_{j=1}^{d} H_{2s_j}(\tilde{\partial}_j T_{\ell;d}(x))\,dx.
		\end{eqnarray*}
	\end{lemma}	
	Before giving the proof of Theorem \ref{thm:laguerre}, we need another technical result.
	
	\begin{lemma}\label{lem:laguerre}
	For $t=(t_1, \dots, t_d)\in \mathbb R^d$,
	 $$
	 \sum_{s_1+...+s_d=p} \frac{1}{s_1!\cdots s_d!} \prod_{j=1}^{d} H_{2s_j}(t_j)= (-1)^{p}2^{p} L_p^{(d/2-1)}\left(\frac{||t||^2}{2}\right),
	 $$
	 where $L_p^{(d/2-1)}$ is the generalized Laguerre polynomial of order $p$ (see Definition \ref{defLaguerre}).
	\end{lemma}

\begin{proof}	
First of all, it is known that Laguerre and Hermite polynomials are related by the following expression
\begin{equation} \label{hermite-laguerre}
H_{2n}(x)=(-1)^n 2^n n! L_n^{(-1/2)}\left(\frac{x^2}{2}\right)
\end{equation}
			(see \cite{abramowitz} (22.5.40) combined with (22.5.18)%anche Szego (5.6.1) ma va sempre riportata ai polinomi di hermite ortonormali
		). This implies that
			\begin{align}\label{hl1}
		\sum_{s_1+\dots+s_d=p} \frac{1}{s_1!\cdots s_d!} \prod_{j=1}^{d} H_{2s_j}(t_j)&= \sum_{s_1+\dots+s_d=p}(-1)^{-p} (2)^{s_1+\dots+s_d}  \frac{s_1!\dots s_d!}{s_1!\cdots s_d!} \prod_{j=1}^{d} L_{s_j}^{(-1/2)}\left(\frac{ t_j^2}{2}\right) \nonumber\\&
			 = (-2)^{p} \sum_{s_1+...+s_d=p}  \prod_{j=1}^{d} L_{s_j}^{(-1/2)}\left(\frac{t_j^2}{2}\right) .
			\end{align}
		Moreover, from the Summation Theorem 8.977, equation (1) on page 1002 of \cite{table},	we have that				
				 \[ L_n^{(\alpha_1+\dots +\alpha_r+r-1)}(x_1+\dots+x_r)= \sum_{m_1+\dots+m_r=n} L_{m_1}^{(\alpha_1)}(x_1)\cdots L_{m_r}^{(\alpha_r)}(x_r), \]		
that applied to (\ref{hl1}) gives
				 \begin{align*}
		\sum_{s_1+...+s_d=p} \frac{1}{s_1!\cdots s_d!} \prod_{j=1}^{d} H_{2s_j}(t_j)&= (-2)^{p} L_{p}^{(d/2-1)} \left(\frac{t_1^2+...+t_d^2}{2}\right) \\&= (-2)^{p} L_{p}^{(d/2-1)} \left(\frac{||t||^2}{2}\right).
				 \end{align*}

\end{proof}
	
We are now in a position to give the proof of our first main result, that is, the Laguerre expansion for the nodal volume.

\begin{proof}[Proof of Theorem \ref{thm:laguerre}]
From Proposition \ref{propChaosHermite}, Lemma \ref{lem:laguerre} together with Lemma \ref{th:hermite} entails that
\begin{eqnarray*}
	\mathcal{L}_{\ell;d-1}[2q]&=& \sqrt{ \frac{\ell(\ell+d-1)}{d} }\sum_{p=0}^{q}\frac{\beta_{2q-2p}}{(2q-2p)!}  \sqrt{2} C(d,p)\\&&
	\times \int_{\mathbb{S}^d} H_{2q-2p}(T_{\ell;d}(x)) L_p^{(d/2-1)}\left(\frac{||\tilde{\nabla} T_{\ell;d}(x)||^2}{2}\right) \,dx.
\end{eqnarray*}
Now (\ref{hermite-laguerre}) applied to $H_{2q-2p}(T_{\ell;d}(x))$ concludes the proof.

%\textcolor{blue}{vedere se riusciamo a dimostrare che ogni caos si comporta come l'integrale di $H_q$ del livello}

\end{proof}

\section{Proof of Proposition \ref{PropVarFormula}}\label{SecVarFormula}

The standard spherical coordinate system on $\mathbb S^d$ is
\begin{equation}\label{x}
x=\left (
\cos \theta_1,
\sin \theta_1 \cos \theta_2,
\sin \theta_1 \sin \theta_2 \cos \theta_3,\dots,
\sin \theta_1 \cdots \sin \theta_{d-1} \cos \varphi,
\sin \theta_1 \cdots \sin \theta_{d-1} \sin \varphi\right )
\end{equation}
where $\theta_i\in [0,\pi]$, $i=1,\dots, d-1$ and $\varphi\in [0,2\pi]$. Note that the geodesic distance between the ``north pole" $(1,0,\dots,0)$ and $x$ in \eqref{x} coincides with $\theta_1$.
The ``meridian" is the set of points such that $\theta_2=\dots = \theta_{d-1}=\varphi =0$.
From now on $\theta=\theta_1$, i.e., we omit the subscript in the first coordinate $\theta_1$. The standard orthonormal basis of the tangent space is given by
\begin{align*}
\frac{\partial}{\partial \theta}, (\sin \theta)^{-1} \frac{\partial}{\partial \theta_2}, \dots, (\prod_{m=1}^{a-1} \sin \theta_m)^{-1} \frac{\partial}{\partial \theta_a}, \dots, (\prod_{m=1}^{d-1}\sin \theta_m)^{-1} \frac{\partial}{\partial \varphi},
\end{align*}
where $a=2,\dots, d-1$.

\subsection{Covariance matrix on the meridian}

This section is inspired by \cite{W10}.
Let $x, y$ be points on the meridian, and $\theta_x, \theta_y$ their spherical coordinates.
We have
\begin{equation*}
\Cov(T_{\ell;d}(x), T_{\ell;d}(y)) = G_{\ell;d}(\cos d(x,y))=G_{\ell;d}(\cos |\theta_x -\theta_y|),
\end{equation*}
and moreover
\begin{align*}
\nabla_x G_{\ell;d}(\cos d(x,y)) = - \nabla_y G_{\ell;d}(\cos d(x,y)) & = (\pm \sin (|\theta_x-\theta_y|) G'_{\ell;d}(\cos |\theta_x - \theta_y|), 0_{d-1})\\
& =:\pm B(|\theta_x - \theta_y|),
\end{align*}
where $\pm$ depends whether or not $\theta_y>\theta_x$, and $0_{d-1}$ denotes the (row) zero vector of dimension $d-1$. For the Hessian we have
\begin{align*}
H_d(d(x,y)) &= G_{\ell;d}''(\cos d(x,y)) \nabla_x \cos (d(x,y)) \otimes \nabla_y \cos (d(x,y)) \\&\qquad+ G_{\ell;d}'(\cos d(x,y)) \nabla_x \otimes \nabla_y \cos d(x,y).
\end{align*}
 %  First of all, let us write the geodesic distance in coordinates for any two points on the sphere
 %  \begin{align*}
%   \cos d(x,y) = \cos \theta_{x,1} \cos \theta_{y,1} \\+
%\sin \theta_{x,1} \cos \theta_{x,2}\sin \theta_{y,1} \cos \theta_{y,2}\\ +
%\sin \theta_{x,1} \sin \theta_{x,2} \cos \theta_{x,3}\sin \theta_{y,1} \sin \theta_{y,2} \cos \theta_{y,3} \\+
%\dots \\+
%\sin \theta_{x,1} \cdots \sin \theta_{x,d-1} \cos \varphi_x\sin \theta_{y,1} \cdots \sin \theta_{y,d-1} \cos \varphi_y\\+
%\sin \theta_{x,1} \cdots \sin \theta_{x,d-1} \sin \varphi_x \sin \theta_{y,1} \cdots \sin \theta_{y,d-1} \sin \varphi_y
%   \end{align*}
Since on the meridian
\begin{align*}
\nabla_x \nabla_y\cos d(x,y) = \begin{pmatrix}
\cos |\phi_x-\phi_y| &0_{d-1}\\
0_{d-1}^t &I_{d-1}
\end{pmatrix},
\end{align*}
where $I_{d-1}$ is the $(d-1)\times (d-1)$ identity matrix, for the Hessian we get
\begin{align*}
& H_d(|\theta_x - \theta_y|)  = G_{\ell;d}''(\cos |\theta_x-\theta_y|) \nabla_x \cos (|\theta_x-\theta_y|) \otimes \nabla_y \cos (|\theta_x-\theta_y|) \\
 &+ G_{\ell;d}'(\cos |\theta_x-\theta_y|) \nabla_x \otimes \nabla_y \cos |\theta_x-\theta_y| \\
%& = G_\ell''(\cos |\theta_x-\theta_y|) (\pm (\sin |\theta_x-\theta_y|,0,\dots,0) \otimes (\mp (\sin |\theta_x-\theta_y|,0,0\dots,0) \\
%&+ G_\ell'(\cos |\theta_x-\theta_y|) \begin{pmatrix}
%\cos |\phi_x-\phi_y| &0 \dots &0\\
%0 &1 &0\dots &0\\
%0 &0 &1 \dots &0\\
%\dots \\
%0 \dots &0 &0  &1
%\end{pmatrix}\\
&= \begin{pmatrix}
G_{\ell;d}'(\cos |\theta_x-\theta_y|)\cos |\phi_x-\phi_y| - G_{\ell;d}''(\cos |\theta_x-\theta_y|)\sin^2 |\theta_x-\theta_y|&0_{d-1}\\
0_{d-1}^t &G_{\ell;d}'(\cos |\theta_x-\theta_y|)I_{d-1}
\end{pmatrix}.
\end{align*}
We have just proved the following.
\begin{lemma}
The covariance matrix $\Sigma$ of the vector $(T_{\ell;d}(x), T_{\ell;d}(y), \nabla T_{\ell;d}(x), \nabla T_{\ell;d}(y))$ for $x$ the north pole and $y$ any point on the meridian such that $d(x,y)=\theta$ is
\begin{equation}\label{covmatrix}
\Sigma = \Sigma(\theta):=\begin{pmatrix}
1 &G_{\ell;d}(\cos \theta) &0_{d} &-B(\theta)\\
G_{\ell;d}(\cos \theta) &1 &B(\theta) &0_{d}\\
0^t_d &B(\theta)^t &\frac{E_{\ell;d}}{d} I_d &H(\theta)\\
-B( \theta)^t &0_d^t &H(\theta) &\frac{E_{\ell;d}}{d} I_d
 \end{pmatrix},
\end{equation}
where $0_{d}$ denotes the (row) zero vector of dimension $d$ and $I_{d}$ is the $d\times d$ identity matrix.
\end{lemma}

\subsection{Diagram formula}

Let $Z=(Z_1, Z_2, \dots, Z_{2d+2})$ be a $(2d+2)$ dimensional centered Gaussian vector, where $\Var(Z_i)=1$ for every $i$. We need to evaluate quantities of the form
$$
\mathbb E \Big [H_{2q-2p}(Z_1)H_{2q-2p'}(Z_2)  \prod_{j=1}^{d} H_{2s_j}(Z_{2+j}))\, \prod_{k=1}^{d} H_{2s'_k}(Z_{d+2+k}) \Big ]
$$ for $q\ge 2$, $p,p'=0,\dots q$, $s_1+\dots + s_d = p$, $s_1'+\dots + s_d'=p'$. The diagram formula (see \cite[Lemma 5.2]{CGR22} and \cite[Section 4.3.1]{MP11}) allows one to write
\begin{align}
&&\mathbb E \Big [H_{2q-2p}(Z_1)H_{2q-2p'}(Z_2)  \prod_{j=1}^{d} H_{2s_j}(Z_{2+j}))\, \prod_{k=1}^{d} H_{2s'_k}(Z_{d+2+k}) \Big ] \nonumber\\
&& = (2q-2p)! (2q-2p')! \prod_{i=1}^d (2s_i)! (2s'_i)! \sum_{\lbrace k_{ij}\rbrace_{i,j=1}^{2d+2}} \prod_{i,j=1, i<j}^{2d+2} \frac{\mathbb E[Z_i Z_j]^{k_{ij}}}{k_{ij}!},\label{diagram}
\end{align}
where
$
k_{ij}\le \min(q_i,q_j), k_{ij}\in \mathbb N, k_{ii}=0, k_{ij}=k_{ji},
$
and moreover
\begin{align*}
 \sum_{j=1}^{2d+2} k_{ij} =\begin{cases}
2q-2p,\qquad i=1\\
2q-2p',\qquad i=2\\
2s_{i-2},\qquad i=3,\dots, d+2\\
2s'_{i-d-2},\qquad i=d+3,\dots, 2d+2.
\end{cases}
\end{align*}
We are now in a position to prove the explicit formula for the variance of chaotic components of the nodal volume.

\begin{proof}[Proof of Proposition \ref{PropVarFormula}]
From Lemma \ref{th:hermite} and Theorem \ref{thm:laguerre} we can write
\begin{eqnarray*}
			&&\Var(\mathcal{L}_{\ell;d-1}[2q]	)\\&&= 	  \frac{\ell(\ell+d-1)}{d} \sum_{p,p'=0}^{q} \frac{\beta_{2q-2p}}{(2q-2p)!}  \frac{(-1)^p\sqrt{2}}{2^p} C(d,p)\frac{\beta_{2q-2p'}}{(2q-2p')!}  \frac{(-1)^{p'}\sqrt{2}}{2^{p'}} C(d,p')\cr
			&&\quad \times \sum_{\substack{s_1+\dots+s_d=p\\ s_1'+\dots + s_d'=p'}} \frac{1}{s_1! \cdots s_d!}\frac{1}{s'_1! \cdots s'_d!}\cr
			&&\quad \times \int_{\mathbb{S}^d}\int_{\mathbb{S}^d}\mathbb E[ H_{2q-2p}(T_{\ell;d}(x))  \prod_{j=1}^{d} H_{2s_j}(\tilde{\partial}_j T_{\ell;d}(x))H_{2q-2p'}(T_{\ell;d}(y))  \prod_{k=1}^{d} H_{2s_k}(\tilde{\partial}_k T_{\ell;d}(y))]\,dxdy\cr
			&&= \frac{\ell(\ell+d-1)}{d} \sum_{p,p'=0}^{q} \frac{\beta_{2q-2p}}{(2q-2p)!}  \frac{(-1)^p\sqrt{2}}{2^p} C(d,p)\frac{\beta_{2q-2p'}}{(2q-2p')!}  \frac{(-1)^{p'}\sqrt{2}}{2^{p'}} C(d,p')\cr &&\quad
			\times \sum_{\substack{s_1+\dots+s_d=p\\ s_1'+\dots + s_d'=p'}} \frac{1}{s_1! \cdots s_d!}\frac{1}{s'_1! \cdots s'_d!}
		 |\mathbb S^d||\mathbb S^{d-1}| \cr
		&& \quad\times  \int_{0}^\pi \mathbb E[ H_{2q-2p}(T_{\ell;d}(\bar x)) H_{2q-2p'}(T_{\ell;d}(y_\theta)) \prod_{j,k=1}^{d} H_{2s_j}(\tilde{\partial}_j T_{\ell;d}(\bar x))   H_{2s_k}(\tilde{\partial}_k T_{\ell;d}(y_\theta))] \sin^{d-1} \theta\,d\theta,
			\end{eqnarray*}
			where $\bar x$ is the north pole and $y_\theta$ is the point on the meridian s.t. $d(\bar x, y_\theta)=\theta$. By applying \eqref{diagram} for $Z=(T_{\ell;d}(\bar x), T_{\ell;d}(y_\theta), \nabla T_{\ell;d}(\bar x), \nabla T_{\ell;d}(y_\theta))$, it follows that
		\begin{eqnarray*}
		\Var(\mathcal{L}_{\ell;d-1}[2q]	)&=&
		\frac{E_{\ell;d}}{d} \sum_{p,p'=0}^{q} \frac{\beta_{2q-2p}}{(2q-2p)!}\frac{\beta_{2q-2p'}}{(2q-2p')!}
	\frac{(-1)^p\sqrt{2}}{2^p}C(d,p)\frac{(-1)^{p'}\sqrt{2}}{2^{p'}}C(d,p')\cr
	&& \times \sum_{s,s' \in \mathbb{N}^d \substack {s_1+\dots+s_d=p\cr
	s'_1+\dots +s'_d=p'}} \frac{1}{s_1!...s_d!}\frac{1}{s'_1!...s'_d!}\times (2q-2p)! (2q-2p')! \prod_{i=1}^d (2s_i)! (2s'_i)! \cr
	&& \times \mathcal H^{d}(\mathbb S^d) \mathcal H^{d-1}(\mathbb S^{d-1})\int_{0}^{\pi} d\theta \sin^{d-1} \theta\sum_{ k_{1,2}+ k_{1,d+3}=2q-2p} \frac{G_{\ell;d}(\cos \theta)^{k_{12}}}{ k_{12}!}\cr
	&& \times  \frac{(-\big ( \frac{\ell(\ell+d-1)}{d} \big )^{-\frac{1}{2}} (\sin \theta)G'_{\ell;d}(\cos \theta))^{ k_{1, d+3}}}{ k_{1,d+3}!}	\cr
	&&\times \sum_{ k_{2,1}+ k_{2,3}=2q-2p'} \frac{( \big ( \frac{\ell(\ell+d-1)}{d} \big )^{-\frac{1}{2}} \sin \theta G'_{\ell;d}(\cos \theta))^{ k_{2, 3}}}{  k_{2,3}!}\cr
	&& \times  \sum_{i=4}^{d+2} \sum_{ k_{i,d+i}=2s_{i-2} } \frac{\Big (\big ( \frac{\ell(\ell+d-1)}{d} \big )^{-1} G'_{\ell;d}(\cos \theta)\Big )^{ k_{i,d+i}}}{ k_{i,d+i}!}\cr
	&& \times \sum_{ k_{2,3}+ k_{3,d+3}=2s_1 } \frac{\Big( \big ( \frac{\ell(\ell+d-1)}{d} \big )^{-1} G_{\ell;d}'(\cos\theta)\cos \theta - G_{\ell;d}''(\cos \theta)\sin^2 \theta\Big )^{ k_{3,d+3}}}{ k_{3,d+3}!}.
			\end{eqnarray*}
Indeed, taking into account the sparsity of the covariance matrix $\Sigma$ in \eqref{covmatrix}, we find the conditions
		\begin{align*}
		\begin{cases}		
		k_{1,2} + k_{1,d+3} = 2q-2p \\
		 k_{1,2} + k_{2,3} = 2q-2p' \\
		 k_{2,3} + k_{3,d+3} = 2s_1 \\
		 k_{i,d+i}=2s_{i-2},\qquad i=4,\dots, d+2 \\
		k_{1,d+3} + k_{3,d+3} = 2s'_1 \\
		 k_{i,d+i} = 2s'_{i-2}, \qquad i=4,\dots, d+2.
		 \end{cases}
		\end{align*}	
			First of all, these conditions imply that
		$$
		s_i = s'_i, \qquad i=2,\dots, d,
		$$
		and moreover
		\begin{equation*}
		k_{1,2} + k_{1,d+3} + k_{2,3} + k_{3,d+3} + \sum_{i=4}^{d+2} k_{i,d+i}  = 2q-2p + 2s_1 + \sum_{i=2}^d 2s_i = 2q-2p +2p = 2q.
		\end{equation*}
		From \cite[(4.5.5)]{szego}, parity of Gegenbauer polynomials, and the above conditions, we can restrict the integration interval from $[0,\pi]$ to $[0,\pi/2]$, up to a factor $2$.
\end{proof}

		\section{Proof of Theorem \ref{thmvar}}
		
		From \cite{Wig09} we write
		\begin{align*}
		\mathbb E[\mathcal L_{\ell;d-1}^2] = \int_{(\mathbb S^d)^2} K_{\ell;d}(x,y)\,dx dy,
		\end{align*}
		where $K_{\ell;d}$ is the two-point correlation function, i..e,
		\begin{equation}
		K_{\ell;d}(x,y) = \frac{1}{(2\pi)^{d/2}\sqrt{1- G_{\ell;d}(\cos d(x,y))^2}}\mathbb E[\| \nabla T_{\ell;d}(x) \| \| \ \nabla T_{\ell;d}(y)\|].
				\end{equation}
				Hence from \eqref{doubleint}
				\begin{align*}
				\Var(\mathcal L_{\ell;d-1}) &= \int_{(\mathbb S^d)^2} (K_{\ell;d}(x,y) - c E_{\ell;d})\,dx dy=\sum_{q=2}^{+\infty} \Var(\mathcal L_{\ell;d-1}[2q]),
				\end{align*}
				where $c$ is s.t. $cE_{\ell;d} \mathcal{H}^{d}(\mathbb{ S}^{d})^2=\mathbb{ E}[\mathcal{L}_{\ell;d-1}]^2,$ i.e.
				\begin{equation*}
				c %=  \frac{1}{\pi d} \frac{\Gamma((d+1)/2)^2}{\Gamma(d/2)^2}
				= \frac{\mathcal{H}^{d-1}(\mathbb{ S}^{d-1})^2}{d \mathcal{H}^{d}(\mathbb{ S}^{d})^2}.
			%	\textcolor{red}{\frac{\mathcal{H}^{d-1}(\mathbb{ S}^{d-1})}{\mathcal{H}^{d}(\mathbb{ S}^{d})d}}.
								\end{equation*}
				By isotropy of the kernel $K$, i.e., $K(x,y)$ only depends on $d(x,y)$, we can define (abusing notation) $K(d(x,y)):=K(x,y)$. Moreover by the change of variable
				$\theta=\psi/L$ where $L$ is still $L= \ell + (d-1)/2$, we can write
				\begin{align*}
				\int_{(\mathbb S^d)^2} K_{\ell;d}(x,y)\,dx dy = &2\mathcal H^d(\mathbb S^d) \mathcal H^{d-1}(\mathbb S^{d-1}) \int_0^{\pi/2} K_{\ell;d}(\theta) \sin^{d-1}\theta\,d\theta\\
				=& 2\mathcal H^d(\mathbb S^d) \mathcal H^{d-1}(\mathbb S^{d-1}) \frac{1}{L}\int_0^{L\pi/2} K_{\ell;d}(\psi/L) \sin^{d-1}\psi/L\,d\psi\\	
				=& 2\mathcal H^d(\mathbb S^d) \mathcal H^{d-1}(\mathbb S^{d-1})	\Big(  \frac{1}{L}\int_0^{C} K_{\ell;d}(\psi/L) \sin^{d-1}\psi/L\,d\psi
				\\
				&+  \frac{1}{L}\int_C^{L\pi/2} K_{\ell;d}(\psi/L) \sin^{d-1}\psi/L\,d\psi\Big ),					
					\end{align*}
					where $C$ is some constant $>0$. (Note that the integration interval is $[0,\pi/2]$, see \cite[Lemma 2.9]{W10}.)
					
					\begin{lemma}[Cf. Lemma 4.4 in \cite{Wig09}] \label{lemIgor}
					As $\ell\to +\infty$,
					\begin{equation*}
					\frac{1}{L}\int_0^{C}( K_{\ell;d}(\psi/L) - c E_{\ell;d})\sin^{d-1}\psi/L\,d\psi = O\left ( \ell^{-(d-2)}\right ).
										\end{equation*}
					\end{lemma}
	\emph{(In order to find the exact asymptotic law for the variance, we believe it is possible to improve Lemma \ref{lemIgor} from $O(\ell^{-(d-2)})$ to $o(\ell^{-(d-2)})$.	)} 				
					
					As for the second summand, we develop a finer analysis based on Proposition \ref{PropVarFormula}. Indeed, let us note first that
					\begin{eqnarray}
										&& \frac{1}{L}\int_C^{L\pi/2} (K_{\ell;d}(\psi/L)- c E_{\ell;d})\sin^{d-1}\psi/L\,d\psi \nonumber \cr
					 & = &\sum_{q=2}^{+\infty}	\frac{E_{\ell;d}}{d} \sum_{p,p'=0}^{q} \frac{\beta_{2q-2p}}{(2q-2p)!}\frac{\beta_{2q-2p'}}{(2q-2p')!}
	\frac{(-1)^p\sqrt{2}}{2^p}C(d,p)\frac{(-1)^{p'}\sqrt{2}}{2^{p'}}C(d,p')\nonumber \cr
	&\times& \sum_{s,s' \in \mathbb{N}^d \substack {s_1+\dots+s_d=p\nonumber \cr
	s'_1+\dots +s'_d=p'}} \frac{1}{s_1!...s_d!}\frac{1}{s'_1!...s'_d!}\times (2q-2p)! (2q-2p')! \prod_{i=1}^d (2s_i)! (2s'_i)! \nonumber \cr
		&\times& \frac{1}{L}\int_{C}^{L\pi/2} d\psi \sin^{d-1}(\psi/L)\sum_{ k_{1,2}+ k_{1,d+3}=2q-2p} \frac{G_{\ell;d}(\cos \psi/L)^{ k_{12}}}{ k_{12}!} \cr
		&\times& \frac{(-\big ( \frac{\ell(\ell+d-1)}{d} \big )^{-\frac{1}{2}} (\sin \psi/L)G'_{\ell;d}(\cos \psi/L))^{ k_{1, d+3}}}{ k_{1,d+3}!}	\nonumber \cr
	&\times& \sum_{ k_{2,1}+ k_{2,3}=2q-2p'} \frac{( \big ( \frac{\ell(\ell+d-1)}{d} \big )^{-\frac{1}{2}} \sin \psi/L G'_{\ell;d}(\cos \psi/L))^{ k_{2, 3}}}{  k_{2,3}!}\nonumber \cr
	&\times&  \sum_{i=4}^{d+2} \sum_{ k_{i,d+i}=2s_{i-2} } \frac{\Big (\big ( \frac{\ell(\ell+d-1)}{d} \big )^{-1} G'_{\ell;d}(\cos \psi/L)\Big )^{ k_{i,d+i}}}{ k_{i,d+i}!}\nonumber\cr
	&&\times \sum_{ k_{2,3}+ k_{3,d+3}=2s_1 } \frac{\Big( \big ( \frac{\ell(\ell+d-1)}{d} \big )^{-1} G_{\ell;d}'(\cos\psi/L)\cos \psi/L - G_{\ell;d}''(\cos \psi/L)\sin^2 \psi/L\Big )^{k_{3,d+3}}}{ k_{3,d+3}!}.\nonumber \\
	\label{identity}
					\end{eqnarray}
					
			Let us now rewrite \eqref{identity} as follows.	Let $0<\varepsilon<1$, adapting Corollary A.3 in \cite{Wig09} for derivatives of Gegenbauer polynomials,
			it is possible to prove that there exists $C>0$ s.t. $\forall \theta \in[C/\ell, \pi/2]$
			\begin{align*}
			&\left |G_{\ell;d}(\cos \theta)\right| < 1-\varepsilon, \\
			&\left |\left (\frac{\ell(\ell+d-1)}{d} \right )^{-\frac{1}{2}} (\sin \theta)G'_{\ell;d}(\cos \theta)\right |< 1 -\varepsilon,\\
			&\left |\left (\frac{\ell(\ell+d-1)}{d} \right )^{-1} G'_{\ell;d}(\cos \theta)\right | < 1-\varepsilon,\\
			&\left | \left (\frac{\ell(\ell+d-1)}{d} \right )^{-1} G_{\ell;d}'(\cos\psi/L)\cos \psi/L - G_{\ell;d}''(\cos \psi/L)\sin^2 \psi/L \right | < 1-\varepsilon.
									\end{align*}
			  Hence we have the following upper bound for \eqref{identity}:
\begin{eqnarray*}
&&\left | \frac{1}{L}\int_C^{L\pi/2} (K_{\ell;d}(\psi/L)- c E_{\ell;d})\sin^{d-1}\psi/L\,d\psi \right |\cr
%&& \le \sum_{q=2}^{+\infty}	\frac{E_{\ell;d}}{d} \sum_{p,p'=0}^{q} \left |\frac{\beta_{2q-2p}}{(2q-2p)!}\frac{\beta_{2q-2p'}}{(2q-2p')!}
	%\frac{(-1)^p\sqrt{2}}{2^p}C(d,p)\frac{(-1)^{p'}\sqrt{2}}{2^{p'}}C(d,p')\right |\nonumber \cr
	%&\times& \sum_{s,s' \in \mathbb{N}^d \substack {s_1+\dots+s_d=p\nonumber \cr
	%s'_1+\dots +s'_d=p'}} \frac{1}{s_1!...s_d!}\frac{1}{s'_1!...s'_d!}\times (2q-2p)! (2q-2p')! \prod_{i=1}^d (2s_i)! (2s'_i)!\cr
	%&&\times  (1-\varepsilon)^{2q-4}\sum_{i=2}^{d}  \frac{\Big (1-\varepsilon\Big )^{2s_{i}}}{(2s_{i})!} \sum_{k\in \mathcal A}\frac{(1-\varepsilon)^{k_{12}}}{k_{12}!} \frac{ (1-\varepsilon)^{k_{1, d+3}}}{k_{1,d+3}!}	 \frac{\big ( 1-\varepsilon)^{k_{2, 3}}}{ k_{2,3}!}
	% \frac{\Big( 1-\varepsilon\Big )^{k_{3,d+3}}}{k_{3,d+3}!}\cr
& \le& \sum_{q=2}^{+\infty}	\frac{E_{\ell;d}}{d} \sum_{p,p'=0}^{q}\left |  \frac{\beta_{2q-2p}}{(2q-2p)!}\frac{\beta_{2q-2p'}}{(2q-2p')!}
	\frac{(-1)^p\sqrt{2}}{2^p}C(d,p)\frac{(-1)^{p'}\sqrt{2}}{2^{p'}}C(d,p')\right |\nonumber \cr
	&&\times \sum_{s,s' \in \mathbb{N}^d \substack {s_1+\dots+s_d=p\nonumber \cr
	s'_1+\dots +s'_d=p'}} \frac{1}{s_1!...s_d!}\frac{1}{s'_1!...s'_d!}\times (2q-2p)! (2q-2p')! \prod_{i=1}^d (2s_i)! (2s'_i)! \nonumber \cr
		&&\times \big ( 1-\varepsilon)^{2q-4}\sum_{i=2}^{d} \frac{1}{(2s_i)!}\sum_{k\in \mathcal A_{q,p,p'}} \frac{1}{k_{12}!} \frac{ 1}{k_{1,d+3}!} \frac{1}{  k_{2,3}!}  \frac{1}{ k_{3,d+3}!}\cr
		&&\times \frac{1}{L}\int_C^{L\pi/2} d\psi \sin^{d-1}\frac{\psi}{L}\sum_{a_1+\dots +a_{d+3}=4} \prod_{j=1}^{d-1}\Big |\Big ( \frac{\ell(\ell+d-1)}{d} \Big )^{-1} G'_{\ell;d}(\cos \psi/L)\Big |^{a_j}\cr
	&&\times \big|G_{\ell;d}(\cos \psi/L)\big|^{ a_d}\times \Big|-\Big ( \frac{\ell(\ell+d-1)}{d} \Big )^{-\frac{1}{2}} (\sin \psi/L)G'_{\ell;d}(\cos \psi/L)\Big|^{ a_{d+1}} \\
	&&\times \Big| \Big ( \frac{\ell(\ell+d-1)}{d} \Big )^{-\frac{1}{2}} \sin \psi/L G'_{\ell;d}(\cos \psi/L)\Big|^{a_{d+2}} \times  \cr
	&&\times \Big| \Big ( \frac{\ell(\ell+d-1)}{d} \Big )^{-1} G_{\ell;d}'(\cos\psi/L)\cos \psi/L - G_{\ell;d}''(\cos \psi/L)\sin^2 \psi/L\Big |^{a_{d+3}}	.		\end{eqnarray*}	

We will also need the following auxiliary result.	
	\begin{lemma}\label{lem4}
	For any $a_1,\dots, a_{d+3}\in \mathbb N$ s.t. $a_1 +\dots + a_{d+3}=4$ we have, as $\ell \to +\infty$,
	\begin{eqnarray*}
	&&\frac{1}{L}\int_C^{L\pi/2} d\psi \sin^{d-1}\frac{\psi}{L} \prod_{j=1}^{d-1}\left |\left( \frac{\ell(\ell+d-1)}{d} \right)^{-1} G'_{\ell;d}(\cos \psi/L)\right|^{a_j}\cr
	&&\times |G_{\ell;d}(\cos \psi/L)|^{ a_d}\times \left |-\left ( \frac{\ell(\ell+d-1)}{d} \right )^{-\frac{1}{2}} (\sin \psi/L)G'_{\ell;d}(\cos \psi/L)\right |^{ a_{d+1}} \cr
	&&\times \left| \left ( \frac{\ell(\ell+d-1)}{d} \right )^{-\frac{1}{2}} \sin \psi/L G'_{\ell;d}(\cos \psi/L)\right |^{a_{d+2}}  \cr
	&&\times \left| \left ( \frac{\ell(\ell+d-1)}{d} \right )^{-1} G_{\ell;d}'(\cos\psi/L)\cos \psi/L - G_{\ell;d}''(\cos \psi/L)\sin^2 \psi/L\right |^{a_{d+3}} = O\Big(\ell^{-d}\Big ).
		\end{eqnarray*}
	\end{lemma}
	The proof of Lemma \ref{lem4} is collected in Section \ref{secLem4}.
	
We can now establish the following bound.					
		\begin{proposition} \label{lemNonSing}
					As $\ell\to +\infty$,
					\begin{equation*}
					\frac{1}{L}\int_C^{L\pi}( K_{\ell;d}(\psi/L) - c E_{\ell;d})\sin^{d-1}\psi/L\,d\psi = O\left ( \ell^{-(d-2)}\right ).
										\end{equation*}
					\end{proposition}			
	\begin{proof}[Proof of Proposition \ref{lemNonSing}]	
Recall that
$$	
\frac{(-1)^p\sqrt{2}}{2^p}C(d,p)\frac{1}{s_1!\dots s_d!} = \frac{\alpha_{2s_1,\dots,2s_d}}{(2s_1)!\dots(2s_d)!}.
$$
The following estimate holds true due to a rough counting of all possible non-flat diagrams in the diagram formula \cite[Proposition 4.15]{MP11}:
\begin{align*}
&(2q-2p)! (2q-2p')! \prod_{i=1}^d (2s_i)! (2s'_i)! 	\sum_{
	k_{1,2}+ k_{1,d+3}=2q-2p} \frac{1}{ k_{12}!} \frac{ 1}{ k_{1,d+3}!} \sum_{ k_{2,1}+ k_{2,3}=2q-2p'} \frac{1}{ k_{2,3}!} \\
&\times \sum_{  k_{2,3}+ k_{3,d+3}=2s_1 }\sum_{i=4}^{d+2} \sum_{ k_{i,d+i}=2s_{i-2} } \frac{1}{k_{i,d+i}!}\ \frac{1}{ k_{3,d+3}!} \le (2q)!.
\end{align*}
Then we can bound
\begin{eqnarray*}
 &&\sum_{q=2}^{+\infty}	\frac{E_{\ell;d}}{d} \sum_{p,p'=0}^{q} \sum_{s,s' \in \mathbb{N}^d \substack {s_1+\dots+s_d=p\nonumber \cr
		s'_1+\dots +s'_d=p'}} \frac{\beta_{2q-2p}}{(2q-2p)!}\frac{\beta_{2q-2p'}}{(2q-2p')!}
\nonumber \cr
&&\quad \times \frac{\alpha_{2s_1,\dots,2s_d}}{(2s_1)!\dots(2s_d)!} \frac{\alpha_{2s'_1,\dots,2s'_d}}{(2s'_1)!\dots(2s'_d)!} \times (2q-2p)! (2q-2p')! \prod_{i=1}^d (2s_i)! (2s'_i)! \nonumber \cr
&&\quad \times \big ( 1-\varepsilon)^{2q-4}\sum_{ k_{1,2}+ k_{1,d+3}=2q-2p} \frac{1}{ k_{12}!} \frac{ 1}{ k_{1,d+3}!} \cr
&&\quad \times\sum_{ k_{2,1}+k_{2,3}=2q-2p'} \frac{1}{ k_{2,3}!} \Big [ \sum_{  k_{2,3}+k_{3,d+3}=2s_1 }\sum_{i=4}^{d+2} \sum_{ k_{i,d+i}=2s_{i-2} } \frac{1}{k_{i,d+i}!}\ \frac{1}{ k_{3,d+3}!}\Big ]\cr
&\leq& \frac{1}{(1-\varepsilon)^4}\sum_{q=2}^{+\infty}	\frac{E_{\ell;d}}{d} (1-\varepsilon)^{2q} (2q)!\cr
 &&\times
 \sum_{p,p'=0}^{q} \sum_{s,s' \in \mathbb{N}^d \substack {s_1+\dots+s_d=p\nonumber \cr
		s'_1+\dots +s'_d=p'}} \frac{\beta_{2q-2p}}{(2q-2p)!}\frac{\beta_{2q-2p'}}{(2q-2p')!}
 \frac{\alpha_{2s_1,\dots,2s_d}}{(2s_1)!\dots(2s_d)!} \frac{\alpha_{2s'_1,\dots,2s'_d}}{(2s'_1)!\dots(2s'_d)!}. \nonumber
	\end{eqnarray*}

We have that

\begin{eqnarray*}
&&\sum_{q=2}^{+\infty}	\frac{E_{\ell;d}}{d} (1-\varepsilon)^{2q} (2q)!\cr
&&\quad \times
\sum_{p,p'=0}^{q} \sum_{s,s' \in \mathbb{N}^d \substack {s_1+\dots+s_d=p\nonumber \cr
		s'_1+\dots +s'_d=p'}} \Big|\frac{\beta_{2q-2p}}{(2q-2p)!}\frac{\beta_{2q-2p'}}{(2q-2p')!}
\frac{\alpha_{2s_1,\dots,2s_d}}{(2s_1)!\dots(2s_d)!} \frac{\alpha_{2s'_1,\dots,2s'_d}}{(2s'_1)!\dots(2s'_d)!} \Big| \nonumber \cr
&\leq& \sum_{q=2}^{+\infty}	\frac{E_{\ell;d}}{d} (1-\varepsilon)^{2q} (2q)!\cr
&&\times
\sum_{s,s' \in \mathbb{N}^d \substack {2s_1+\dots+2s_d+2c=2q\nonumber \cr
		2s'_1+\dots +2s'_d+2c'=2q}} \Big|\frac{\beta_{2c}}{(2c)!}\frac{\alpha_{2s_1,\dots,2s_d}}{(2s_1)!\dots(2s_d)!}\Big|\Big
|	\frac{\beta_{2c'}}{(2c')!}
 \frac{\alpha_{2s'_1,\dots,2s'_d}}{(2s'_1)!\dots(2s'_d)!} \Big |\nonumber \cr
 &\leq &
\frac{E_{\ell;d}}{d} \sum_{\substack{2s_1,\dots,2s_d,2s'_1,\dots,2s'_d,2c,2c'\geq 0 \\ 2s_1+\dots+2s_d+2c \geq 2 \\ 2s'_1+\dots+2s'_d+2c' \geq 2 }} \Big
 |\frac{\beta_{2c}}{(2c)!}\frac{\alpha_{2s_1,\dots,2s_d}}{(2s_1)!\dots(2s_d)!}\Big|
 \Big|	\frac{\beta_{2c'}}{(2c')!}
 \frac{\alpha_{2s'_1,\dots,2s'_d}}{(2s'_1)!\dots(2s'_d)!} \Big|
 \cr
 &&\times \sqrt{(1-\varepsilon)}^{2s_1+\dots+2s_d+2c+2s'_1+\dots+2s'_d+2c'} \sqrt{(2s_1+\dots+2s_d+2c)!} \sqrt{(2s'_1+\dots+2s'_d+2c')!} \nonumber \cr
 &\leq& \frac{E_{\ell;d}}{d} \sum_{\substack{2s_1,\dots,2s_d,2s'_1,\dots,2s'_d,2c,2c'\geq 0 \\ 2s_1+\dots+2s_d+2c \geq 2 \\ 2s'_1+\dots+2s'_d+2c' \geq 2 }}
\Big |\frac{\beta_{2c}}{(2c)!}\frac{\alpha_{2s_1,\dots,2s_d}}{(2s_1)!\dots(2s_d)!} \Big|^2 (2s_1+\dots+2s_d+2c)!
 \cr
 &&\times \sqrt{(1-\varepsilon)}^{2s_1+\dots+2s_d+2c+2s'_1+\dots+2s'_d+2c'}
\end{eqnarray*}
due to Cauchy-Schwarz applied as in the proof of Lemma 3.5 \cite{DNPR19}.

Since the map
$(2s_1,\dots,2s_d,2c) \to \frac{\alpha^2_{2s_1,\dots,2s_d}\beta_{2c}^2}{(2s_1)!\dots(2s_d)!(2c)!}$ is bounded uniformly over $\varepsilon$, we have
\begin{eqnarray*}
&\leq& \frac{E_{\ell;d}}{d} C  \sum_{\substack{2s_1,\dots,2s_d,2s'_1,\dots,2s'_d,2c,2c'\geq 0 \\ 2s_1+\dots+2s_d+2c \geq 2 \\ 2s'_1+\dots+2s'_d+2c' \geq 2 }}
\frac{(2s_1+\dots+2s_d+2c)!}{(2s_1)!\dots(2s_d)!(2c)!}
\sqrt{(1-\varepsilon)}^{2s_1+\dots+2s_d+2c+2s'_1+\dots+2s'_d+2c'}.
\end{eqnarray*}

Moreover, estimating $\frac{(2s_1+\dots+2s_d+c)!}{(2s_1)!\dots(2s_d)!(2c)!} \leq (d+1)^{2s_1+\dots+2s_d+2c}$ and choosing $\varepsilon$ such that $(d+1) \sqrt{1-\eps}<1$,
we conclude that
\begin{eqnarray*}
&\leq& \frac{E_{\ell;d}}{d} C  \sum_{\substack{2s_1,\dots,2s_d,2s'_1,\dots,2s'_d,2c,2c'\geq 0 \\ 2s_1+\dots+2s_d+2c \geq 2 \\ 2s'_1+\dots+2s'_d+2c' \geq 2 }}
(d+1)^{2s_1+\dots+2s_d+2c}
\sqrt{(1-\varepsilon)}^{2s_1+\dots+2s_d+2c+2s'_1+\dots+2s'_d+2c'} <\infty.
\end{eqnarray*}

Finally, from Lemma \ref{lem4} we have
\begin{eqnarray*}
&&\frac{1}{L}\int_C^{L\pi/2} d\psi \sin^{d-1}\frac{\psi}{L}\sum_{a_1+\dots a_{d+3}=4} \prod_{j=1}^{d-1}\Big |\left ( \frac{\ell(\ell+d-1)}{d} \right )^{-1} G'_{\ell;d}(\cos \psi/L)\Big |^{a_j}\cr
	&&\quad \times \left|G_{\ell;d}(\cos \psi/L)\right|^{ a_d}\times \Big|-\left ( \frac{\ell(\ell+d-1)}{d} \right )^{-\frac{1}{2}} (\sin \psi/L)G'_{\ell;d}(\cos \psi/L) \Big|^{ a_{d+1}} \cr
	&&\quad \times \Big| \left( \frac{\ell(\ell+d-1)}{d} \right )^{-\frac{1}{2}} \sin \psi/L G'_{\ell;d}(\cos \psi/L) \Big|^{a_{d+2}} \times  \cr
	&&\quad\times \Big| \left ( \frac{\ell(\ell+d-1)}{d} \right )^{-1} G_{\ell;d}'(\cos\psi/L)\cos \psi/L - G_{\ell;d}''(\cos \psi/L)\sin^2 \psi/L\Big |^{a_{d+3}} = O\left(\frac{1}{\ell^d} \right)			
\end{eqnarray*}
thus concluding the proof.
\end{proof}

\section{Proof of Lemma \ref{lem4}}\label{secLem4}

To prove Lemma \ref{lem4}   we first need the following asymptotics for powers of Gegenbauer polynomials, whose proof follows applying the expansion of Gegenbauer polynomials and its derivatives given in Lemma \ref{lem:G} .

%\subsection{Preliminaries}
		Let us set for notational simplicity,
		\begin{equation}\label{Xi}
		\Xi_{\ell;d}(\psi) :=  \frac{2^{d/2-1}}{{\ell + d/2-1 \choose \ell}} \frac{\Gamma(\ell+d/2)}{L^{d/2-1}\ell!} \frac{1}{(\sin\psi/\ell)^{d/2-1}}.
				\end{equation}
	\begin{lemma}\label{lemPowers} As $\ell\to +\infty$, uniformly for $\psi > C$, for $k>0, k \in \mathbb{ N}$,
	\begin{align*}
	(G_{\ell;d}(\cos(\psi/L)))^{k} %=&\Big ( \Psi(d,\ell,\psi) \Big [ \sqrt{\frac{2}{\pi \ell \sin \psi /L}} \Big ( \sin(\psi - (d/2-1)\pi/2 +\pi/4) \\
 % &+ \frac{(4(d/2-1)^2 - 1) }{8} \frac{\cos(\psi - (d/2-1)\pi/2 +\pi/4)}{\psi}\Big )
  % + O(\frac{1}{\psi^{5/2}}) + O(1/\ell \sqrt \psi)    \Big ] \Big )^{\tilde k_{1,2}} \\
   =&  \Big [\Xi_{\ell;d}(\psi) \sqrt{\frac{2}{\pi \ell \sin \psi /L}} \Big ( \sin\big(\psi - (\frac{d}{2}-1)\frac{\pi}{2} +\frac{\pi}{4}\big)\Big )\Big ]^{ k} \\
   &+  O\left(\frac{1}{\psi^{2+(d-1){k}/2}}+\frac{1}{\ell \psi^{(d-1)k/2}}\right)
     		\end{align*}
		Moreover, 	
		\begin{align*}
&\Big (-\Big( \frac{\ell(\ell+d-1)}{d} \Big )^{-\frac{1}{2}} \sin \frac{\psi}{L} G'_{\ell;d}(\cos 	\psi/L) \Big )^{ k} \\
&= \Big [\Big( \frac{\ell(\ell+d-1)}{d} \Big )^{-\frac{1}{2}} \frac{\sqrt{\ell}}{\sin^{1/2} \psi/L} \Xi_{\ell;d}(\psi) \sqrt{\frac{2}{\pi}}
\cos\big(\psi - (\frac{d}{2}-1)\frac{\pi}{2} +\frac{\pi}{4}\big)\Big]^{ k} \\
&\quad  + O\left(\frac{\ell^{3{k}-2}}{\psi^{(3/2+d/2){k}-1}}+\frac{\ell^{3{k}-1}}{\psi^{(3/2+d/2){k}+1}}\right),
\end{align*}
and %Furthermore we have the following \textcolor{red}{negligible} contribution
\begin{align*}
 &\Big (\Big ( \frac{\ell(\ell+d-1)}{d} \Big)^{-1} G'_{\ell;d}(\cos \frac{\psi}{L})\Big )^{k} \\
 &= \Big [-\Big( \frac{\ell(\ell+d-1)}{d} \Big )^{-1} \frac{\sqrt{\ell}}{\sin^{3/2} \psi/L}\Xi_{\ell;d}(\psi) \sqrt{\frac{2}{\pi}}
\cos\big(\psi - (\frac{d}{2}-1)\frac{\pi}{2} +\frac{\pi}{4}\big)  \Big ]^{k} \\
& \quad
+ O\left(\frac{\ell^{3k-2}}{\psi^{(5/2+d/2)k-1}}+\frac{\ell^{3k-1}}{\psi^{(5/2+d/2)k+1}}\right).
 \end{align*}
Finally %since the contribution of $G'_\ell$ is negligible w.r.t. the other terms,
\begin{align*}
&\Big( \Big( \frac{\ell(\ell+d-1)}{d} \Big)^{-1} \Big (G_{\ell;d}'(\cos\frac{\psi}{L})\cos \frac{\psi}{L} - G_{\ell;d}''(\cos \frac{\psi}{L})\sin^2 \frac{\psi}{L}\Big)\Big )^{k}\\
%&= \big ( \frac{\ell(\ell+d-1)}{d} \big )^{-2s_1}\Big ( G_\ell'(\cos\frac{\psi}{L})\cos \frac{\psi}{L} -((d-1)+1) G_\ell'(\cos\psi/L)+\ell^2(1+\frac{(d-1)}{\ell})G_\ell(\cos\psi/L)+error\Big )^{2s_1}\\
& = \left ( \frac{\ell(\ell+d-1)}{d} \right )^{-k}\Big ( \ell^2\Big(1+\frac{d-1}{\ell}\Big)G_{\ell;d}(\cos\psi/L)\Big )^{k} +  O\left(\frac{1}{\psi^{2+(d-1)k/2}}+\frac{1}{\ell\psi^{(d-1)k/2}}\right).
\end{align*}
\end{lemma}

\begin{proof}[Proof of Lemma \ref{lem4}] Let us consider the integral
  \begin{eqnarray*}
  	I:&=&\frac{1}{L}\int_C^{L\pi/2} d\psi \sin^{d-1}\frac{\psi}{L} \prod_{j=1}^{d-1}\left |\left( \frac{\ell(\ell+d-1)}{d} \right)^{-1} G'_{\ell;d}(\cos \psi/L)\right|^{a_j}\cr
  	&&\times |G_{\ell;d}(\cos \psi/L)|^{ a_d}\times \left |-\left ( \frac{\ell(\ell+d-1)}{d} \right )^{-\frac{1}{2}} (\sin \psi/L)G'_{\ell;d}(\cos \psi/L)\right |^{ a_{d+1}} \cr
  	&&\times \left| \left ( \frac{\ell(\ell+d-1)}{d} \right )^{-\frac{1}{2}} (\sin \psi/L) G'_{\ell;d}(\cos \psi/L)\right |^{a_{d+2}}  \cr
  	&&\times \left| \left ( \frac{\ell(\ell+d-1)}{d} \right )^{-1} G_{\ell;d}'(\cos\psi/L)\cos \psi/L - G_{\ell;d}''(\cos \psi/L)\sin^2 \psi/L\right |^{a_{d+3}}, %= O\left (\frac{1}{\ell^d}\right ).
  \end{eqnarray*}

  for any $a_1,\dots, a_{d+3}\in \mathbb N$ s.t. $a_1 +\dots + a_{d+3}= 4$.

  Exploiting the asymptotics in Lemma \ref{lemPowers}, it follows that
  \begin{align*}
  I&= \left ( \frac{2^{d/2-1}}{(d/2-1)!}\right )^{4} \left( \frac{2}{\pi}\right )^{2} (-1)^{a_{d+1}}d^{\frac{a_{d+1}+a_{d+2}}{2} + a_{d+3}+\sum_{j=1}^{d-1} a_j}\\
  &\quad \times  \frac{1}{L^d}\int_{0}^{L\pi/2} d\psi \frac{1}{\psi^{d-1+\sum_{j=1}^{d-1} a_j}} \sin^{a_{d}+a_{d+3}}\big(\psi - (\frac{d}{2}-1)\frac{\pi}{2} +\frac{\pi}{4}\big)\\
  &\quad \times	\cos^{a_{d+1}+a_{d+2} +\sum_{j=1}^{d-1} a_j}\big(\psi - (\frac{d}{2}-1)\frac{\pi}{2} +\frac{\pi}{4}\big)+ o\left(\frac{1}{\ell^d}\right).
  \end{align*}

  Note that the following integral
  \begin{equation*}
\int_{0}^{\infty} d\psi \frac{1}{\psi^{d-1+\sum_{j=1}^{d-1} a_j}} \sin^{a_{d}+a_{d+3}}\big(\psi - (\frac{d}{2}-1)\frac{\pi}{2} +\frac{\pi}{4}\big)\cos^{a_{d+1}+a_{d+2} +\sum_{j=1}^{d-1} a_j}\big(\psi - (\frac{d}{2}-1)\frac{\pi}{2} +\frac{\pi}{4}\big)
  \end{equation*}
  is absolutely convergent for every $d\geq3$, hence the order of magnitude of $I$ is $\frac{1}{L^d}$, thus concluding the proof.

 % Let us consider now $I_1.$ Denoting again $\theta=\psi/L$ we obtain

%  \begin{align*}
%  &I_1= \frac{1}{L}\int_{0}^{C} d\psi\sin^{d-1} \frac{\psi}{L} \sum_{i=2}^d\frac{\Big (\big ( \frac{\ell(\ell+d-1)}{d} \big )^{-1} G'_\ell(\cos \frac{\psi}{L})\Big )^{2s_{i}}}{(2s_i)!}\\
%  &\times \sum_{ k\in \mathbb N^{4} : k\in \mathcal A_{4,p,p'}} \frac{G_\ell(\cos \frac{\psi}{L})^{ k_{12}}}{ k_{12}!}\times \frac{(-\big ( \frac{\ell(\ell+d-1)}{d} \big )^{-\frac{1}{2}} (\sin \frac{\psi}{L})G'_\ell(\cos \frac{\psi}{L}))^{ k_{1, d+3}}}{ k_{1,d+3}!} \times \frac{( \big ( \frac{\ell(\ell+d-1)}{d} \big )^{-\frac{1}{2}} \sin \frac{\psi}{L} G'_\ell(\cos \frac{\psi}{L}))^{\ k_{2, 3}}}{  k_{2,3}!} \times  \\
%  &\times \frac{\Big( \big ( \frac{\ell(\ell+d-1)}{d} \big )^{-1} G_{\ell;d}'(\cos\frac{\psi}{L})\cos \frac{\psi}{L} - G_{\ell,d}''(\cos \frac{\psi}{L})\sin^2 \frac{\psi}{L}\Big )^{ k_{3,d+3}}}{ k_{3,d+3}!},
%  \end{align*}

\end{proof}

\appendix

\section{Asymptotics for Gegenbauer polynomials}\label{AppendixA}

This section generalizes to higher dimensions some results for $d=2$ given by \cite[Lemma B.3]{W10}; in the sequel, recall definition (\ref{Xi}).

	\begin{lemma}\label{lem:G}
	For $\ell\ge 1$ and $\psi > C$ we have
	\begin{eqnarray}
	G_{\ell;d}(\cos\frac{\psi}{L}) &=& \Xi_{\ell;d}(\psi)\times \Big\{ \sqrt{\frac{2}{\pi \ell \sin \psi /L}} \Big [ \sin\big(\psi - (\frac{d}{2}-1)\frac{\pi}{2} +\frac{\pi}{4}\big) \nonumber \cr
  &&+ \frac{(4(d/2-1)^2 - 1) }{8\psi} \cos\big(\psi - (\frac{d}{2}-1)\frac{\pi}{2} +\frac{\pi}{4}\big)\Big ]
   + O\left(\frac{1}{\psi^{5/2}}\right) \nonumber \cr&&+ O\left(\frac{1}{\ell \sqrt \psi}\right)    \Big \};\nonumber \cr
  \label{G0} \cr
			G'_{\ell;d}(\cos \frac{	\psi}{L})
			&=& \frac{\sqrt{\ell}}{\sin^{5/2} \psi/L} \Xi_{\ell;d}(\psi) \sqrt{\frac{2}{\pi}}\bigg[
			-
			\cos\big(\psi - (\frac{d}{2}-1)\frac{\pi}{2} +\frac{\pi}{4}\big)\sin(\frac{\psi}{\ell})
		\nonumber 	\cr&&+ \frac{d^2-1}{8l}\sin\big(\psi - (\frac{d}{2}-1)\frac{\pi}{2} +\frac{\pi}{4}\big) \bigg]\nonumber
			+O\left( \frac{\ell}{\psi^{d/2+1/2}}\right) \cr&&
			+O\left(\frac{\ell^{2}}{\psi^{d/2+5/2}}\right);  \label{G1}\\
			G_{\ell;d}''(\cos \frac{\psi}{L})&=&\frac{d}{\sin^2(\psi/L)} G_{\ell;d}'(\cos\frac{\psi}{L})-\frac{\ell^2}{\sin^2(\psi/L)}\left(1+\frac{d-1}{\ell}\right)G_{\ell;d}(\cos\frac{\psi}{L}) \nonumber \cr
			&&
			+O\Big(\frac{\ell^3}{\psi^{d/2+3/2}}\Big) .\label{G2}
		\end{eqnarray}
		
	\end{lemma}
	\begin{proof} Let us start with (\ref{G0}).
		From Hilb's asymptotic formula, recalled in (\ref{hilbs}), we can write
	\begin{align*}
(\sin \frac{\psi}{\ell})^{d/2-1} G_{\ell;d}(\cos\frac{\psi}{L}) &= \frac{2^{d/2-1}}{{\ell + d/2-1 \choose \ell}} \Big ( \frac{\Gamma(\ell+d/2)}{L^{d/2-1}\ell!} \sqrt{\frac{\psi/L}{\sin \psi/L}} J_{d/2-1}(\psi) + O\left(\frac{\sqrt \psi}{\ell^{2}}\right)       \Big )
%\\
%&=  \frac{2^{d/2-1}}{{\ell + d/2-1 \choose \ell}} \Big ( \sqrt{\frac{\psi/L}{\sin \psi/L}} J_{d/2-1}(\psi) + O(\sqrt \psi/\ell^{2})       \Big ).
	\end{align*}
which leads to
	\begin{align}\label{G}
 G_{\ell;d}(\cos\frac{\psi}{L}) = 	\Xi_{\ell;d}(\psi) \Big ( \sqrt{\frac{\psi/L}{\sin \psi/L}} J_{d/2-1}(\psi) + O\left(\frac{\sqrt \psi}{\ell^{2}}\right)       \Big ).
    	\end{align}
    	
We now exploit the Hankel's expansion for Bessel functions \cite[pp. 237-242]{Olv97}:
    	\[J_{\nu}\left(z\right)\sim\left(\frac{2}{\pi z}\right)^{\frac{1}{2}}\*\left(%
    	\cos\omega\sum_{k=0}^{\infty}(-1)^{k}\frac{a_{2k}(\nu)}{z^{2k}}-\sin\omega\sum%
    	_{k=0}^{\infty}(-1)^{k}\frac{a_{2k+1}(\nu)}{z^{2k+1}}\right),\]			
    	where, for $k\ge 1$,
    	\[a_{k}(\nu)=\frac{(4\nu^{2}-1^{2})(4\nu^{2}-3^{2})\cdots(4\nu^{2}-(2k-1)^{2})}{%
    		k!8^{k}}=\frac{{\left(\frac{1}{2}-\nu\right)_{k}}{\left(\frac{1}{2}+\nu\right)%
    			_{k}}}{(-2)^{k}k!},\]
    	with the convention $a_0(\nu)=1$. In particular, in our case $\nu=d/2-1$ and then
    	\begin{equation*}
    	a_0(d/2-1) = 1, \qquad a_1(d/2-1) = \frac{(4(d/2-1)^2 - 1) }{8}.
    		\end{equation*}
Thus we find
	\begin{align*}
	J_{d/2-1}(\psi)
	=&  \sqrt{\frac{2}{\pi \psi}}  \sin\big(\psi - (\frac{d}{2}-1)\frac{\pi}{2} +\frac{\pi}{4}\big)  %a_0(d/2-1)
	\\
	&+\sqrt{\frac{2}{\pi \psi}}  \cos\big(\psi - (\frac{d}{2}-1)\frac{\pi}{2} +\frac{\pi}{4}\big) \frac{(4(d/2-1)^2 - 1) }{8} \frac{1}{\psi} 	+ O\left(\frac{1}{\psi^{5/2}}\right).
				\end{align*}

Substituting this expansion in (\ref{G}) we get

\begin{eqnarray*}
 G_{\ell;d}(\cos\frac{\psi}{L}) &=& 	\Xi_{\ell;d}(\psi)
\Big \{ \sqrt{\frac{\psi/L}{\sin \psi/L}}\Big [ \sqrt{\frac{2}{\pi \psi}}  \sin\big(\psi - (\frac{d}{2}-1)\frac{\pi}{2} +\frac{\pi}{4}\big) %a_0(d/2-1)
	\\
	&&+\sqrt{\frac{2}{\pi \psi}}  \cos\big(\psi - (\frac{d}{2}-1)\frac{\pi}{2} +\frac{\pi}{4}\big) \frac{(4(d/2-1)^2 - 1) }{8} \frac{1}{\psi} 	+ O\left(\frac{1}{\psi^{5/2}}\right) \Big ] \cr
  && + O\left(\frac{\sqrt \psi}{\ell^{2}}\right)       \Big\}  \cr
  &=& \Xi_{\ell;d}(\psi)
\times \Big \{\sqrt{\frac{2}{\pi \ell \sin \psi /L}} \Big [ \sin\big(\psi - (\frac{d}{2}-1)\frac{\pi}{2} +\frac{\pi}{4}\big)\cr
  &&+ \frac{(4(d/2-1)^2 - 1) }{8\psi} \cos\big(\psi - (\frac{d}{2}-1)\frac{\pi}{2} +\frac{\pi}{4}\big)\Big ]  + O\left(\frac{1}{\psi^{5/2}}\right) + O\left(\frac{1}{\ell \sqrt \psi}\right)    \Big \},
      	\end{eqnarray*}
where we used that $
	O(\sqrt \psi/\ell^2)	=  O(1/\ell \sqrt \psi)
$
	since $\psi /\ell$ is bounded, being equal to $\theta$. \\

Let us now  prove (\ref{G1}).
Recall the definition of Jacobi polynomials $P_\ell^{(\alpha,\beta)}$, given in Section \ref{subsec:sphericalharmonics}. From (4.5.1) in \cite{szego}, it is known that
\begin{align*}
(2\ell + \alpha +\beta) (1-x^2) &\frac{d}{dx}P^{(\alpha, \beta)}_\ell(x) \\ %= {\ell +d/2 -1\choose \ell} \Big (  \frac{(2\ell+d-2)}{2}  G_{\ell-1}(\cos \psi/L) -\ell \cos(\psi/L)  G_{\ell}(\cos \psi/L)      \Big ),\\
&= -\ell[ (2\ell+\alpha +\beta)x + \beta -\alpha ]P_\ell^{(\alpha, \beta)}(x) + 2(\ell+\alpha)(\ell+\beta) P_{\ell-1}^{(\alpha,\beta)}(x),
\end{align*}
with $P_\ell^{(\alpha, \beta)}(1) = {\ell +\alpha \choose \ell}$.
Since $P^{(d/2-1, d/2-1)}_\ell (x)=  {\ell +d/2-1 \choose \ell} G_{\ell;d}(x)$,
we derive
\begin{align*}
&  \, G'_{\ell;d}(\cos 	\frac{\psi}{L})= \frac{\ell}{\sin^2 \psi/L} \Big [ (1 + \frac{d-2}{2\ell})  \frac{{\ell +d/2-2 \choose \ell-1}}{{\ell +d/2-1 \choose \ell}}G_{\ell-1;d}(\cos \frac{\psi}{L}) -  \cos(\frac{\psi}{L})  G_{\ell;d}(\cos \frac{\psi}{L})   \Big ].
\end{align*}
Then in view of (\ref{G0}) we get
\begin{eqnarray*}
  G'_{\ell;d}(\cos 	\frac{\psi}{L})
&=& \frac{\ell}{\sin^2 \psi/L} \Big \{ (1 + \frac{d-2}{2\ell})  \frac{{\ell +d/2-2 \choose \ell-1}}{{\ell +d/2-1 \choose \ell}} \cr
&&
\times \Big[ \frac{\Gamma(\ell+d/2-1)}{(L-1)^{d/2-1}(\ell-1)!}\frac{2^{d/2-1}}{{\ell + d/2-2 \choose \ell-1}} \frac{1}{(\sin\psi/\ell)^{d/2-1}} \sqrt{\frac{2}{\pi (L-1)\sin \psi /L }} \cr
&&
\times \Big ( \sin\big(\psi - (\frac{d}{2}-1)\frac{\pi}{2} +\frac{\pi}{4}\big)\cr
&&+ \frac{(4(d/2-1)^2 - 1) }{8\psi(L-1)/L} \cos\big(\psi - (\frac{d}{2}-1)\frac{\pi}{2} +\frac{\pi}{4}\big)\Big )
\cr&&+ O\left(\frac{1}{\psi^{5/2}\psi^{d/2-1}}\right) + O\left(\frac{\sqrt \psi}{\ell^2} \frac{1}{\psi^{d/2-1}}\right)   \Big]\cr
&&-  \cos(\frac{\psi}{L})  \Big[\frac{2^{d/2-1}}{{\ell + d/2-1 \choose \ell}} \frac{\Gamma(\ell+d/2)}{L^{d/2-1}\ell!}  \frac{1}{(\sin\psi/\ell)^{d/2-1}} \sqrt{\frac{2}{\pi \ell \sin \psi /L}} \cr&&\Big ( \sin\big(\psi - (\frac{d}{2}-1)\frac{\pi}{2} +\frac{\pi}{4}\big)
+ \frac{(4(d/2-1)^2 - 1) }{8\psi} \cos\big(\psi - (\frac{d}{2}-1)\frac{\pi}{2} +\frac{\pi}{4}\big)\Big )\cr&&
+ O\left(\frac{1}{\psi^{5/2}\psi^{d/2-1}}\right) + O\left(\frac{\sqrt \psi}{\ell^2} \frac{1}{\psi^{d/2-1}}\right)   \Big]\Big \}.
\end{eqnarray*}
Handling the computations and using that $\frac{\sin(\psi/L)}{\psi/L}=1+O(\psi^2/L^2)$ and $\cos(\psi/L)=1+O(\psi^2/L^2)$ we find the expression
\begin{align*}
 \, G'_{\ell;d}(\cos 	\frac{\psi}{L})
&= \frac{\sqrt{\ell}}{\sin^{5/2} \psi/L} \Xi_{\ell;d}(\psi)\sqrt{\frac{2}{\pi}} \bigg[
-
\cos\big(\psi - (\frac{d}{2}-1)\frac{\pi}{2} +\frac{\pi}{4}\big)\sin(\frac{\psi}{\ell})
\\&\quad+ \frac{d^2-1}{8l}\sin \big(\psi - (\frac{d}{2}-1)\frac{\pi}{2} +\frac{\pi}{4}\big) \bigg]
+O\left(\frac{1}{\psi^{1/2}} \frac{1}{\psi^{d/2-1}}\right) +O\left( \frac{L}{\psi^{3/2}} \frac{1}{\psi^{d/2-1}}\right) \\&\quad
+O\left(\frac{L}{\psi^{5/2}} \frac{1}{\psi^{d/2-1}} +\frac{L^{2}}{\psi^{7/2}} \frac{1}{\psi^{d/2-1}}\right)
\\&\quad
+ O\left(\frac{\ell^3}{\psi^{9/2}\psi^{d/2-1}}\right) + O\left( \frac{1}{\psi^{d/2-1}}\frac{L}{ \psi^{3/2}}\right) .
\end{align*}

Noting that
$O\left(\frac{1}{\psi^{1/2}} \frac{1}{\psi^{d/2-1}}\right) =O\left( \frac{L}{\psi^{3/2}} \frac{1}{\psi^{d/2-1}}\right)$,
  $O\left(\frac{L}{\psi^{5/2}} \frac{1}{\psi^{d/2-1}}\right)=O\left(\frac{L^{2}}{\psi^{7/2}} \frac{1}{\psi^{d/2-1}}\right)$
and
 $O\left(\frac{\ell^3}{\psi^{9/2}\psi^{d/2-1}}\right) =O\left(\frac{L}{\psi^{3/2}\psi^{d/2-1}}\right)$, we conclude that
\begin{eqnarray*}
  \, G'_{\ell;d}(\cos \frac{	\psi}{L})
&=& \frac{\sqrt{\ell}}{\sin^{5/2} \psi/L} \Xi_{\ell;d}(\psi) \sqrt{\frac{2}{\pi}} \bigg[
-
\cos\big(\psi - (\frac{d}{2}-1)\frac{\pi}{2} +\frac{\pi}{4}\big)\sin(\frac{\psi}{\ell})
\\&&+ \frac{d^2-1}{8l}\sin \big(\psi - (\frac{d}{2}-1)\frac{\pi}{2} +\frac{\pi}{4}\big) \bigg]
+O\left( \frac{L}{\psi^{3/2}} \frac{1}{\psi^{d/2-1}}\right) \\&&
+O\left(\frac{L^{2}}{\psi^{7/2}} \frac{1}{\psi^{d/2-1}}\right)
\end{eqnarray*}
which leads to (\ref{G1}).\\

Finally, to prove (\ref{G2}) we remind that
	the Gegenbauer polynomials are particular solutions of the differential equation:
	\[ \sin^2(\frac{\psi}{L})G_{\ell;d}''(\cos\frac{\psi}{L})=\Big(2\big(\frac{d}{2}-\frac{1}{2}\big)+1\Big)\cos \frac{\psi}{L} G_{\ell;d}'(\cos\frac{\psi}{L})-\ell\Big(\ell+2\big(\frac{d}{2}-\frac{1}{2}\big)\Big)G_{\ell;d}(\cos\frac{\psi}{L}) .\]
	%	\[ \sin^2(\psi/L)G_{\ell;d}''(\cos \psi/L)=((d-1)+1)\cos(\psi/L) G_{\ell;d}'(\cos\psi/L)-\ell(\ell+(d-1))G_{\ell;d}(\cos\psi/L) \]
From whom we get
		\begin{eqnarray*}
		G_{\ell;d}''(\cos \frac{\psi}{L})&=&\frac{((d-1)+1)}{\sin^2(\psi/L)} G_{\ell;d}'(\cos\frac{\psi}{L})-\frac{\ell^2}{\sin^2(\psi/L)}\left(1+\frac{d-1}{\ell}\right)G_{\ell;d}(\cos\frac{\psi}{L}) \\ && +O\Big(\frac{\ell^3}{\psi^{5/2}\psi^{d/2-1}}\Big)
		\end{eqnarray*}
	which concludes the proof.
\end{proof}

\medskip

Finally we are in the position to prove Corollary \ref{BerryCancel}.

\begin{proof}[Proof of Corollary \ref{BerryCancel}]
Recall that we denote by $\mathcal{L}_{\ell;d-1}(u)$ the volume of the set $T_{\ell;d}^{-1}(u)$; let us write $\mathcal{L}_{\ell;d-1}(u)[2]$ for its second chaotic projection.
From an adaptation of the proof of  \cite[Lemma 4.1]{C19} to the spherical case and the same argument as in the proof of Theorem \ref{thm:laguerre}, with $q=1$, it can be easily shown that
\begin{eqnarray*}
	\mathcal{L}_{\ell;d-1}(u)[2]&=& \sqrt{ \frac{\ell(\ell+d-1)}{d} }\sum_{p=0}^{1}\frac{\beta_{2-2p}(u)}{(2-2p)!}  \sqrt{2} C(d,p)\\&&
	\times \int_{\mathbb{S}^d} H_{2-2p}(T_{\ell;d}(x)) L_p^{(d/2-1)}\left(\frac{||\tilde{\nabla} T_{\ell;d}(x)||^2}{2}\right) \,dx,
\end{eqnarray*}
where $\beta_{m}(u):=\frac{e^{-u^2/2}}{\sqrt{2\pi}} H_m(u)$. %$\phi(\cdot)$ being as usual the density function of a standard Gaussian random variable.
Note that $C(d,0)=\frac{\Gamma(\frac{d+1}{2})}{\Gamma(\frac{d}{2})}$, $C(d,1)=- \frac{\Gamma(\frac{d+1}{2})}{2\Gamma(\frac{d+2}{2})}$ and  $\beta_0(u)=\frac{e^{-u^2/2}}{\sqrt{2\pi}}$, $\beta_2(u)=\frac{e^{-u^2/2}}{\sqrt{2\pi}}(u^2-1)$. Then
\begin{eqnarray*}
	\mathcal{L}_{\ell;d-1}(u)[2]	&=& \sqrt{ \frac{\ell(\ell+d-1)}{d} } \frac{e^{-u^2/2}}{2\sqrt{\pi}} \Gamma(\frac{d+1}{2})\frac{1}{\Gamma(d/2)}\\&&
\times \Big ( (u^2-1) \int_{\mathbb{S}^d} H_{2}(T_{\ell;d}(x))\,dx-\frac{1}{d}  \int_{\mathbb{S}^d} d- ||\tilde{\nabla}T_{\ell;d}(x)||^2 \,dx\Big ).
\end{eqnarray*}
Since $||\tilde{\nabla}T_{\ell;d}(x)||^2=\langle \tilde{\nabla} T_{\ell;d}(x), \tilde{\nabla} T_{\ell;d}(x) \rangle$, applying  Green's identity (see also \cite{R19}), using that $T_{\ell;d}$ satisfies the Helmholtz equation, and writing $\frac{\Gamma((d+1)/2)}{\Gamma(d/2)}=\frac{\sqrt{\pi}\mathcal{H}^{d-1}(\mathbb{S}^{d-1})}{\mathcal{H}^{d}(\mathbb{S}^{d})}$, we find

\begin{eqnarray*}
	\mathcal{L}_{\ell;d-1}(u)[2]	&=& \sqrt{ \frac{\ell(\ell+d-1)}{d} }  \frac{\sqrt{\pi}\mathcal{H}^{d-1}(\mathbb{S}^{d-1})}{\mathcal{H}^{d}(\mathbb{S}^{d})}
 \frac{ e^{-u^2/2}	u^2}{2\sqrt{\pi}}\int_{\mathbb{S}^d} H_{2}(T_{\ell;d}(x))\,dx.
\end{eqnarray*}

It is known that

$$\int_{0}^{\pi} G_{\ell;d}^2(\cos\theta) \sin^{d-1}\theta \,d\theta=\frac{\mathcal{H}^d(\mathbb{ S}^d)}{\mathcal{H}^{d-1}(\mathbb{ S}^{d-1})\eta_{\ell;d}}$$

and then
$$
\text{Var}\left(\int_{\mathbb{S}^d} H_{2}(T_{\ell;d}(x))\,dx\right)=\frac{2\mathcal{H}^d(\mathbb{ S}^d)^2}{\eta_{\ell;d}} \sim  \mathcal{H}^d(\mathbb{ S}^d)^2(d-1)!
\frac{1}{\ell^{d-1}},\qquad \ell\to +\infty.
$$
The result now follows immediately.
\end{proof}

\end{document}